\documentclass[a4paper,12pt]{amsart}
\font\s=cmb9
\usepackage{latexsym}
\usepackage{amssymb}
\usepackage{graphicx}
\usepackage{amsthm}
\usepackage{epsfig}
\RequirePackage{color}
\usepackage[normalem]{ulem}
\usepackage{float}
\usepackage{calrsfs}
\DeclareMathAlphabet{\pazocal}{OMS}{zplm}{m}{n}

\usepackage{amscd,enumerate}
\usepackage{amsfonts}

\input xy
\xyoption{all}

\usepackage{multicol}
\usepackage{hyperref}
\hypersetup{ colorlinks,
linkcolor=blue,
filecolor=green,
urlcolor=blue,
citecolor=blue}

\newtheorem{theorem}{Theorem}[section]
\newtheorem{lemma}[theorem]{Lemma}
\newtheorem{corollary}[theorem]{Corollary}
\newtheorem{proposition}[theorem]{Proposition}

\theoremstyle{definition}
\newtheorem{definition}[theorem]{Definition}

\newtheorem{remark}[theorem]{Remark}

\numberwithin{equation}{section}

\def\Z{\mathbb Z}   
\def\N{\mathbb N}
\def\clz{\mathop{\mathcal C}}
\def\T{\mathop{\tau}}
\def\path{\mathop{\hbox{Path}}}
\def\LPA{L_\K(E) }
\def\path{\mathop{\hbox{Path}} }
\def\K{{\mathbb K}}
\definecolor{turquoise2}{rgb}{0,0.898039,0.933333}
\definecolor{magenta}{rgb}{1,0,1}
\def\L{\mathop{\hbox{$\Lambda$}}}
\def\a{\alpha}
\def\b{\beta}
\def\m{\mu}
\def\l{\lambda}

\begin{document}

\title{On the centroid  of a Leavitt path algebra}

\author[D. Gon\c{c}alves]{Daniel Gon\c{c}alves}
\address{Departamento de Matem\'{a}tica - UFSC - Florian\'{o}polis - SC, Brazil}
\email{daemig@gmail.com}
\thanks{D. Gon\c{c}alves was partially supported by Conselho Nacional de Desenvolvimento Cient\'ifico e Tecnol\'ogico (CNPq) grant numbers 304487/2017-1 and 406122/2018-0  and Capes-PrInt grant number 88881.310538/2018-01 - Brazil. The last three authors are supported by the Junta de Andaluc\'{\i}a  through projects  FQM-336 and UMA18-FEDERJA-119 and  by the Spanish Ministerio de Ciencia e Innovaci\'on   through project  PID2019-104236GB-I00,  all of them with FEDER funds.}

\author[D. Mart\'{\i}n]{Dolores Mart\'{\i}n Barquero}
\address{Departamento de Matem\'atica Aplicada, Universidad de M\'alaga, Espa\~na. }
%\curraddr{Department of Mathematics and Statistics,
%Case Western Reserve University, Cleveland, Ohio 43403}
\email{dmartin@uma.es}

\author[C. Mart\'{\i}n]{C\'andido Mart\'{\i}n Gonz\'alez}
\address{Departamento de \'Algebra, Geometr\'{\i}a y Topolog\'{\i}a, Universidad de M\'alaga, Espa\~na.}
%\curraddr{Department of Mathematics and Statistics,
%Case Western Reserve University, Cleveland, Ohio 43403}
\email{candido\_m@uma.es}

\author[M. Siles]{Mercedes Siles Molina}
\address{Departamento de \'Algebra, Geometr\'{\i}a y Topolog\'{\i}a, Universidad de M\'alaga, Espa\~na.}
%\curraddr{Department of Mathematics and Statistics,
%Case Western Reserve University, Cleveland, Ohio 43403}
\email{msilesm@uma.es}

%    General info

%\date{January 1, 2018 and, in revised form, June 22, 2018.}

%\dedicatory{This paper is dedicated to my cat.}


\begin{abstract}
We describe the centroid  of some Leavitt path algebras. More precisely, we show that for Leavitt path algebras over a field $\K$ that are simple its centroid is isomorphic to $\K$, and for prime Leavitt path algebras its centroid is isomorphic to  $\K$  except if the graph is a row-finite comet, in which case the centroid is isomorphic to $\K[x,x^{-1}]$. \end{abstract}

\maketitle

\section{Introduction}

The center of an algebra and its generalizations (as the centroid and extended centroid) play an important role in ring theory. For example, the center is preserved by Morita equivalence, that is, if two algebras are Morita equivalent their centers are isomorphic (this follows from the fact that the center of an algebra can be derived from the category of its modules). For strongly prime rings (see \cite[Theorem 2.1]{Kau}) it is known that Morita equivalent rings have isomorphic extended centroids (see \cite[Theorem 2.5]{Kau}). But as far as we know a definitive answer to the relation between the centroids of Morita equivalent (non-unital) rings is still not clear. 
The center also plays a key role in determining simplicity of (partial) skew group rings, see \cite{BGOD, DG, GOR, Oinert}. When studying prime rings satisfying a generalized polynomial identity the extended centroid plays an important role, see \cite{Mart}. In Lie theory , the centroid appears in a number of papers about affine Lie algebras and also in relation to triple systems (\cite{BN}, \cite{P}). It has been also intensively studied in Jordan theory (\cite{mccrimmon}) and recently in the theory of evolution algebras (\cite{labra}).

The description and study of the center of Leavitt path algebras has started with the description of the center for simple Leavitt path algebras (\cite{GK}), followed by the description for prime Leavitt path algebras (\cite{CMMSS1}, for row finite algebras (\cite{CMMSS}), and finally for arbitrary Leavitt path algebras (\cite{BH}). In \cite{CMMS} the center is studied using the description of Leavitt path algebras as Steinberg algebras (see \cite{Ben} for the definition of Steinberg algebras). 

As we will see below the center of a Leavitt path algebra associated to a finite graph coincides with its centroid. But for infinite graphs very often the center of a Leavitt path algebra is zero, while the centroid is not. Therefore the centroid carry aditional information about the algebra that can not be extracted from the center, and it is relevant to describe then for Leavitt path algebras.

We organize our work as follows. After this brief introduction, we include a section of preliminaries and notation regarding Leavitt path algebras and centroids. We also include in this section some initial results that will be used along the text. In Section~3 we compute the centroid of a simple Leavitt path algebra and we devote Section~4 to a study of centroids via direct (inverse) limits, which we apply to describe the centroid of certain Leavitt path algebras. We study the centroid of the remaining prime Leavitt path algebras in Section~5. Finally, in Section~6 we summarize our results in the main theorem of the paper.

\section{Preliminaries}

\subsection{Leavitt path algebras}
 A \emph{directed graph} consists of a $4$-tuple $E=(E^0, E^1, r_E, s_E)$ 
consisting of two disjoint sets $E^0$, $E^1$ and two maps
$r_E, s_E: E^1 \to E^0$. The elements of $E^0$ are the \emph{vertices} and the elements of 
$E^1$ are the \emph{edges} of $E$.  Further, for $e\in E^1$, $r_E(e)$ and $s_E(e)$ are 
called the \emph{range} and the \emph{source} of $e$, respectively. 
If there is no confusion with respect to the graph we are considering, we simply write $r(e)$ and $s(e)$.

A vertex $v$ such that $s^{-1}(v) = \emptyset$ is called a \emph{sink};
$v$ is called an \emph{infinite emitter} if $s^{-1}(v)$ is an infinite set.  Otherwise, a vertex
 that is neither a sink nor an infinite emitter is called a \emph{regular vertex}.
The  set of infinite emitters will be denoted by ${\rm Inf}(E)$ while ${\rm Reg}(E)$ will denote the set of regular vertices.

In order to define the Leavitt path algebra, we need to introduce the  
{\it extended graph of} $E$.  This is the graph 
$\widehat{E}=(E^0,E^1\cup (E^1)^*, r_{\widehat{E}}, s_{\widehat{E}}),$ where
$(E^1)^*=\{e_i^* \ | \ e_i\in  E^1\}$ and the functions $r_{\widehat{E}}$ and $s_{\widehat{E}}$ are defined as 
\[{r_{\widehat{E}}}_{|_{E^1}}=r,\ {s_{\widehat{E}}}_{|_{E^1}}=s,\
r_{\widehat{E}}(e_i^*)=s(e_i), \hbox{ and }  s_{\widehat{E}}(e_i^*)=r(e_i).\]
The elements of $E^1$ { are } called \emph{real edges}, while {for each  $e\in E^1$ we} call $e^\ast$ a
\emph{ghost edge}.    

A nontrivial \emph{path} $\mu$ in a graph $E$ is a finite sequence of edges { $\mu=e_1\dots e_n$
such that $r(e_i)=s(e_{i+1})$} for $i=1,\dots,n-1$.

In this case, {$s(\mu):=s(e_1)$ and $r(\mu):=r(e_n)$} are the
\emph{source} and \emph{range} of $\mu$, respectively, and $n$ is the \emph{length} of $\mu$, denoted $|\mu|$. We also say that
$\mu$ is {\emph{a path from $u:=s(e_1)$ to $v:=r(e_n)$}} and write $u\geq v$. 
We  {write  $\mu^0$ for} the set of the vertices which are sources or ranges of the edges appearing in the expression of $\mu$, i.e.,
$\mu^0:=\{s(e_1),r(e_1),\dots,r(e_n)\}$. 

We view an element $v$ of $E^{0}$ as a path of length $0$. In this case $s(v)=r(v)=v$. 
The set of all (finite) paths of a graph $E$ is denoted by ${\rm Path}(E)$.  We will use also the notation $\hbox{Path}(E)^*$ with the meaning 
$\hbox{Path}(E)^*:=\{\lambda^*\ \vert \ \lambda\in \hbox{Path}(E)\}$.
We define a \emph{walk} in $\LPA$ as an element
of the form $\a\b^*$ where $\a,\b\in\path(E)$ with $r(\a)=r(\b)$. If $\mu$ is a path in $E$, and if $v=s(\mu)=r(\mu)$, then $\mu$ is called a \emph{closed path based at $v$}.
If $s(\mu)=r(\mu)$ and $s(e_i)\neq s(e_j)$ for every $i\neq j$, then $\mu$ is called a \emph{cycle}. An edge $e$ is an {\it exit} for a path $\mu = e_1 \dots e_n$ if there exists $i\in \{1, \dots, n\}$ such that
$s(e)=s(e_i)$ and $e\neq e_i$.

In this paper we will consider $\N=\{0, 1,\dots\}$ and $\K$ will denote a field.

 An \emph{infinite path} is an infinite sequence of edges $\lambda=f_0f_1\dots$ such that
$r(f_i)=s(f_{i+1})$ for every $i\in \N$.  We define the source of an infinite path to be $s(\lambda):=s(f_1)$.
 We denote the set of all infinite paths by $E^{\infty}$. 

\medskip
%%%%%%%%%%%%%%%%%%%%%%%%%%%%%%%%%%%%%%%%%%%%%%%%%%%

%\noindent \textbf{Leavitt path algebras.}   
Given a (directed) graph $E$ and a commutative unital ring  $R$, the {\it path $R$-algebra} of $E$,
denoted by $RE$, is defined as the free associative $R$-algebra generated by the
set of paths of $E$ with relations:
\begin{enumerate}
\item[(V)] $vw= \delta_{v,w}v$ for all $v,w\in E^0$.
\item [(E1)] $s(e)e=er(e)=e$ for all $e\in E^1$.
\end{enumerate}
The {\it Leavitt path algebra of} $E$ {\it with coefficients in} $R$, denoted $L_R(E)$, 
is the quotient of the path algebra $R\widehat{E}$ by the ideal of $R\widehat{E}$ generated by the relations:
\begin{enumerate}    
\item[(CK1)] $e^*e'=\delta _{e,e'}r(e) \ \mbox{ for all } e,e'\in E^1$.
\item[(CK2)] $v=\sum _{\{ e\in E^1\mid s(e)=v \}}ee^* \ \ \mbox{ for every}\ \ v\in  {\rm Reg}(E).$
\end{enumerate}
Observe that in $R\widehat{E}$ the relations (V) and (E1) remain valid and that the following is also satisfied:
\begin{enumerate}
\item [(E2)] $r(e)e^*=e^*s(e)=e^*$ for all $e\in E^1$.
\end{enumerate}
It is not difficult to show that
\[L_R(E) = \operatorname{span} \{\alpha\beta^{\ast } \ \vert \ \alpha, \beta \in {\rm Path}(E)\}.\]
and that $L_{R}(E)$ is a $\mathbb{Z}$-graded $R$-algebra, where
for each $n\in\mathbb{Z}$, the degree $n$-component $L_{R}(E)_{n}$ is spanned by the set
\[\{\alpha \beta^{\ast }\ \vert \  \alpha, \beta \in {\rm Path}(E)\ \hbox{and}\  |\alpha|-|\beta|=n\}.\] 

%\textcolor{blue}{The referee would rather try to decrypt a hieroglyphic of the ancient imperial Egypt, than guessing the meaning of the below remark.}
%
%\begin{remark} 
%\rm Our notation for elements of $L_R(E)$ and elements of 
%the set ${\rm Path}(E)$ is the same.  Note that we will often be %considering elements of ${\rm Path}(E)$
%where we will not assume any of the structure that comes with %elements of the quotient 
%$L_R(E)$.
%\end{remark}
%
%\textcolor{blue}{Here is an alternative redaction, though not %deserving fireworks either...}\medskip
%
%{\noindent \bf Alternative Remark 1.1} 
%Our notation for elements of $\path(E)$ and 
%for the corresponding elements seen within the algebra $L_R(E)$ %is the same. We will specify the kind of elements we are dealing %with, should we need to distinguish them. \medskip

For vertices $u, v$ we say $u \ge v$ whenever there is a path $\mu$ such that $s(\mu)=u$ and $r(\mu)=v$, 
A subset $H$ of $E^0$ is called \emph{hereditary} if $v\ge w$ and $v\in H$ imply $w\in H$. A
hereditary set is \emph{saturated} if every regular vertex which feeds into $H$ and only into $H$ is again
in $H$, that is, if $s^{-1}(v)\neq \emptyset$ is finite and $r(s^{-1}(v))\subseteq H$ imply $v\in H$. Given a vertex $u\in E^0$, the \emph{tree} of $u$, denoted by $T(u)$, is the set:
$$T(u)= \{v \in E^0 \ \vert u \geq v\}.$$
For a subset of vertices $X$, the \emph{tree} of $X$, denoted by $T(X)$, is the set:
$$T(X)=\bigcup_{u\in X}T(u).$$

For a set $X\subseteq E^0$, the \emph{hereditary and saturated closure of} $X$, denoted $\overline X$, is the smallest hereditary and saturated subset of $E^0$ containing $X$. It can be described in the following way (see \cite[Lemma 2.0.7]{AAS}): Let $\Lambda^0(X) := T(X)$ and $\Lambda^{n+1}(X):=\{v \in \text{Reg}(E^0): r(s^{-1}(v))\in \Lambda^n(X)\}\cup \Lambda^n(X)$. Then $\overline{X}=\cup_{n\geq 0} \Lambda^n(X)$. If there is no confusion with respect to the set $X$ we are considering, we simply write $\Lambda^{n}$.
The set of all hereditary and saturated subsets of $E^0$ is denoted by $\mathcal{H}_E$.
Following \cite[Definition 3.2]{AAS1}, we say that a graph $E$ is a \emph{comet} if it has exactly one cycle $c$ with $T (v) \cap c^0\ne \emptyset$ for every vertex $v \in E^0$, and every infinite path ends in the cycle $c$.
For the definitions not included in the paper, we refer the reader to \cite{AAS}. 
\medskip

\begin{remark}\label{basetipica}
Consider a Leavitt path algebra $L_\K(E)$ and fix a basis $B$ as specified in \cite[Theorem 1, Section 3]{Kirillov} or \cite[Corollary 1.5.12]{AAS}. An element $z=\sum_i k_i\a_i\b_i^*\in L_\K(E)$, where $k_i\in \K$ and the elements $\a_i\b_i^*$ belong to $B$, will be said to be written in {\em normal form relative to} $B$. We also speak about the {\em normal expression} of $z$. 
\end{remark}

\begin{definition}
Consider a Leavitt path algebra $L_\K(E)$ and fix a basis $B$ as explained in Remark \ref{basetipica}. Define $\partial_B\colon L_\K(E)\to\N$
as follows: 
\begin{enumerate}[\rm (i)]
    \item $\partial_B(\a)=\hbox{length}(\a)$ for $\a\in\path(E)$,
    \item $\partial_B(\a\b^*)=\partial_B(\a)$ for any $\a\b^*\in B$, 
    \item$\partial_B(z)=\max(\partial_B(\a_i\b_i^*))$, where $z=\sum k_i \a_i\b_i^*$ is the normal expression of $z$ relative to $B$.
\end{enumerate}
\end{definition}

\subsection{The centroid of an algebra}
Recall that the \emph{centroid} of a $\K$-algebra $A$, denoted by $\clz(A)$, is the $\K$-vector space
of all linear maps $\T\colon A\to A$ such that 
$$\T(xy)=\T(x)y=x\T(y),$$
for any $x,y\in A$. The centroid is also a $\K$-algebra under composition. In particular, if $A^2=A$ then $\clz(A)$ is commutative. The elements of $\clz(A)$  will be called \emph{centralizers}. For any centralizer $\T\in \clz(A)$ its kernel and image are both ideals of $A$. Thus, in the case of a simple algebra, each $\T\in \clz(A)$ is invertible and the centroid $\clz(A)$ is a field.
Notice that there is a homomorphism $L\colon Z(A)\to \clz(A)$ such that $a\mapsto L_a$ (the left multiplication operator), where $Z(A)$ denotes the center of $A$. If $\hbox{Lann}(A)=0$ then the previous map $L$ is a monomorphism (recall that $\hbox{Lann}(A):=\{x\in A\colon xA=0\}$), and if $A$ is unital then $L$ is actually an isomorphism. In particular for the Leavitt path algebra associated to a finite graph, we have that
\begin{equation}\label{ugna}
Z(L_\K(E))=\clz(L_\K(E)),
\end{equation}
but for infinite graphs, very frequently we encounter the situation
$Z(L_\K(E))=0$ and $\clz(L_\K(E))\ne 0$.

%The \emph{extended centroid} of a semiprime $\K$-algebra $A$ is the center of its maximal algebra of quotients, equivalently, it coincides with the center of the maximal right (left) algebra of quotients of $A$ (see \cite[Section 2.3]{BMM}). It can be described as the set of equivalence classes $(\tau, I)$, where $I$ is an essential ideal of $A$ and $\tau:I \to A$ is a map satisfying $\T(xy)=\T(x)y=x\T(y)$ for all $x, y \in I$. Recall that two elements $(\tau, I), (\mu, J)$ are said to be in the same class if $\tau$ and $\mu$ coincides on an ideal of $A$ contained in $I\cap J$. 

\begin{proposition}{\rm (First extension property).}\label{platano}
Let $E$ be an arbitrary graph and $\K$ a field.  
Let $\T\colon E^0\to L_\K(E)$ be a map satisfying
$\T(s(f))f=f \T(r(f))$ and $\T(r(f))f^*=f^* \T(s(f))$ for any $f\in E^1$. Then there is a unique centralizer of $L_\K(E)$ whose
restriction to $E^0$ is $\T$.
\end{proposition}
\begin{proof} By induction it is easy to prove that, for any $\lambda, \mu \in {\rm Path}(E)$, % with $r(\lambda)=r(\mu)$:
$$\T(s(\lambda))\lambda=\lambda \T(r(\lambda)) $$
$$\T(r(\lambda))\lambda^* =\lambda^* \T(s(\lambda)).$$
From here it follows that, for any walk $\l\m^\ast$,
$$\T(s(\lambda))\lambda\mu^*=\lambda\mu^* \T(s(\mu)).$$
Taking into account that any element of $L_\K(E)$ can be written, in normal form relative to a fixed basis, as a linear combination of walks, by linearity we may extend $\tau$ to
$\sigma\colon L_\K(E)\to L_\K(E)$ as the linear map such that for any element $\lambda\mu^*$ in the above basis $\sigma(\lambda\mu^*):=\T(s(\lambda))\lambda\mu^*$.
Now, it is not difficult to check that $\sigma\in \clz(L_\K(E))$ and $\sigma\vert_{E^0}=\T$. 
\end{proof} 

\medskip

%\textcolor{magenta}{Check the Lemma and Corollary below: they %have been added in early September, after heavy (but short) %showers in Alhaur\'{\i}n, (this seems the weather forecast!!).}

\begin{lemma}\label{rap}
Let $E$ be a graph and $\K$ an arbitrary field. Assume $u\in E^0$ is not the base of a cycle and fix a basis $B$ of $L_\K(E)$ as in Remark \ref{basetipica}. For $z\in Z(uL_\K(E)u)$ we have a normal expression $z=ku+\sum f\xi_f f^*$, $k\in \K$, (the sum being extended to all $f\in s^{-1}(u)$ and the $\xi_f$'s are elements in $Z\left(r(f)L_\K(E)r(f)\right)$, all of them zero except for finitely many of them).
\end{lemma}
\begin{proof}
If $u$ is a sink or $\vert s^{-1}(u)\vert=1$ the result is obvious.
Otherwise we have a normal expression  $z=ku+\sum_{f,g} f\xi_{f,g} g^*$ where $k\in \K$, and $\xi_{f,g}\in L_\K(E)$ are zero except for finitely many of them. In this expression we have taken into account that u is not the base of a cycle, hence, no more summands appear. Moreover $f,g\in s^{-1}(u)$. Since $z$ is in the center of the corner $uL_\K(E)u$ we have $z hh^*=hh^* z$ for any $h\in s^{-1}(u)$. Thus $k\ hh^*+\sum_f f\xi_{f,h} h^*=k\  hh^*+\sum_g h\xi_{h,g} g^*$ which gives
$$h(k+\xi_{h,h})h^*+\sum_{f\ne h} f\xi_{f,h}h^*=
h(k+\xi_{h,h})h^*+\sum_{g\ne h} h\xi_{h,g}g^*.$$
Thus $\sum_{f\ne h} f\xi_{f,h}h^*=\sum_{g\ne h} h\xi_{h,g}g^*$. Multiplying on the right by $h$ we get that $\sum_{f\ne h} f\xi_{f,h}=0$, for any $h\in s^{-1}(u)$. Therefore
$$z=ku+\sum_{f} f\xi_{f,g} f^*+\sum_g(\sum_{f\ne g} f\xi_{f,g}) g^*=ku+\sum_{f} f\xi_{f} f^*,$$
where $\xi_{f}:=\xi_{f,f}$. Finally, to prove that $\xi_f\in Z(r(f)L_\K(E)r(f))$ take an arbitrary element $w\in r(f)L_\K(E)r(f)$. Then $fwf^*\in uL_\K(E)u$ and hence
$z fwf^*=fwf^* z$. So $(ku+\sum_{h} h\xi_{h} h^*)fwf^*=fwf^*(ku+\sum_{h} h\xi_{h} h^*)$. This implies $f\xi_fwf^*=fw\xi_f f^*$, giving $\xi_fw=w\xi_f$.
\end{proof}

\begin{corollary}\label{paru}
Let $E$ be an acyclic graph and $\K$ an arbitrary field. Fix a basis $B$ of $L_\K(E)$ as in Remark~ \ref{basetipica}. If $z\in Z(uL_\K(E)u)$ then we have a normal expression $z=\sum_i k_i \a_i\a_i^*$ with $k_i\in \K$, and $\a_i\in\path(E)$ with $s(\alpha_i)=u$ for any $i$. In particular $Z(L_\K(E))\subseteq Z(L_\K(E))_0$, the homogeneous component of degree $0$ of $Z(L_\K(E))$.
\end{corollary}

\begin{proof}
Fix a basis $B$ of $L_\K(E)$ as in Remark \ref{basetipica}. Write $\partial:=\partial_B$ and 
consider an element $z\in Z(uL_\K(E)u)$. If $\partial(z)=0$ we have 
$z=k u$ for some scalar $k$. In this case we are done.
Assume the results holds for any $z$ with $\partial(z)<n$. Take $z$ with $\partial(z)=n$.
By Lemma~\ref{rap} we have a normal form 
$z=ku+\sum f\xi_f f^*$, where $\partial(\xi_f)<n$ for any $f\in s^{-1}(u)$. Since $\xi_f$ is in the center of $r(f)L_\K(E)r(f)$, the induction hypothesis implies that each $\xi_f$ has a normal expression of the form $\sum k_i\a_i\a_i^*$.
Replacing $\xi_f$ by its normal expression we get a normal expression for $z$ as required.
Finally, observe that $Z(L_\K(E))\subseteq Z(L_\K(E))_0$ follows taking into account the obtained expression. This containment can be also derived from the structure theorem of the center of a Leavitt path algebra (see \cite[Theorem 3.3]{CMMS}).
\end{proof}

Recall that a graph is said to satisfy \emph{Condition} (MT3), also called \emph{downward directedness}, if given vertices $v$ and $w$, there exists a vertex $u$ such that $v\geq u$ and $w\geq u$. See \cite[Definition 4.1.2]{AAS} for details.

\begin{proposition}\label{criacaballar}
Let $E$ be an acyclic graph which satisfies Condition (MT3). If  $u\in E^0$ and $z\in Z(uL_\K(E)u)$, then $z=ku$ for some $k\in \K.$
\end{proposition}
\begin{proof}
Assume first that $u$ is an infinite emitter. We use induction on the number $\partial_B$ for a fixed basis $B$ of $L_\K(E)$ as in Remark \ref{basetipica}. If $\partial_B(z)=0$ we are done. Assume $\partial_B(z)=1$.
Write $z=k u+\sum_i k_i f_if_i^*$ in the basis $B$, where $k,k_i\in\K$ and some $k_i$ is nonzero. Fix this $i$. Since $u$ is an infinite emitter, it is possible to find an edge $f_j$ which does not appear in the expression of $z$. Then by Condition (MT3) the vertices $r(f_i)$ and $r(f_j)$ connect to a certain vertex $w$. So, choose paths $\a$ from $r(f_i)$ to $w$ and $\b$ from $r(f_j)$ to $w$. Then $f_i\a\b^*f_j^*\in uL_\K(E)u$ and therefore $zf_i\a\b^*f_j^*=f_i\a\b^*f_j^* z$. This commutativity implies $k_i=0$, which is a contradiction.
Now assume that the property holds for any $z$ satisfying the hypothesis and such that $\partial_B(z)<n$. Take $z$ with 
$\partial_B(z)=n$ and consider a normal expression $z=k u+\sum f\xi_f f^*$ relative to $B$. 
We have $\partial_B(\xi_f)<n$ for any $f$
and by Lemma \ref{rap} each element $\xi_f$ is in the center of the corresponding corner. Thus, by the induction hypothesis, $\xi_f= k_f r(f)$ for some $k_f\in\K$. Consequently $z=k u+\sum_f k_f ff^*$ so that $\partial_B(z)\le 1$ and then $z$ is a scalar multiple of $u$ as required.

If $u$ is a sink we are done. Finally, assume then that $u$ is neither an infinite emitter nor a sink, and let $z$ be in $Z(uL_\K(E)u)$. Write again $z$ in its normal form relative to a fixed basis $B$. To prove the result we use induction on the number $\partial_B(z)$. Assume that $f_0$ is one of the \lq\lq forbidden\rq\rq\ elements in $B$ (i.e. $f_0$ is one element of the form $e_{n_\nu}^\nu$ in \cite[Corollary 1.5.12]{AAS}) and satisfies $f_0f_0^*\in uAu$.
If $\partial_B(z)=0$ we are done. Let us deal with the case $\partial_B(z)=1$.
Then $z=k u+\sum_{i=1}^n k_i f_if_i^*$, where $k,k_i\in\K$ and some $k_i$ is nonzero. Fix $i$ such that $k_i\neq 0$. By Condition (MT3) the ranges $r(f_i)$ and $r(f_0)$ connect to a certain vertex $w$, so there are paths $\a$ from $r(f_i)$ to $w$ and $\b$ from $r(f_0)$ to $w$. Then $f_i\a\b^*f_0^*\in uL_\K(E)u$ and therefore $zf_i\a\b^*f_0^*=f_i\a\b^*f_0^* z$. This commutativity implies $k_i=0$, which is a contradiction.
Now assume that the property hold for any $z$ in the hypothesis with $\partial_B(z)<n$. Take an $z$ with 
$\partial_B(z)=n$. We can write $z=k u+\sum f\xi_f f^*$ a normal expression of $z$ relative to $B$. 
We have $\partial_B(\xi_f)<n$ for any $f$
and by Lemma \ref{rap} each element $\xi_f$ is in the center of the corresponding corner. Thus by the induction hypothesis $\xi_f= k_f r(f)$ for some $k_f\in\K$. Consequently $z=k u+\sum_f k_f ff^*$ so that $\partial_B(z)\le 1$ and therefore $z$ is a scalar multiple of $u$ as required.
\end{proof}

\section{Computing the centroid of a simple Leavitt path algebra}

For an arbitrary $\K$-algebra $A$ and  an idempotent $u$ of $A$, we define:
$$C_u:=\{\T(u)\ \vert \ \T\in \clz(A)\},$$
which is a a subalgebra of $uAu$. It is well known that when $A$ is simple its centroid is a field. For simple algebras, $C_u$ will also be a field, as shown in the lemma that follows.

\begin{lemma}\label{peassomulo}
Let $A$ be a $\K$-algebra and $u$ an idempotent in $A$. Then:
\begin{enumerate}[\rm (i)]
\item $C_u$ is contained in the center of $uAu$.
\end{enumerate}

Assume that $A$ is simple. Then:

\begin{enumerate}
\item[\rm (ii)] $C_u$ is a field. 
\item[\rm (iii)] For $x, y\neq 0$, with $x\in C_u$ and $y\in uAu$, we have $xy\ne 0$.
\end{enumerate}
\end{lemma}
\begin{proof}
(i) Take $\T(u)\in C_u$ and $x\in uAu$. Then $\T(u)x=\T(ux)=\T(x)=\T(xu)=x\T(u)$.

(ii) Take $0\ne \T(u)\in C_u$. Since $A$ is simple,  $\clz{(A)}$ is a field, hence there exists $\T^{-1}$ and we have $\T^{-1}(u)\T(u)=\T^{-1}(u\T(u))=\T^{-1}\T(u^2)=u.$
Thus $\T(u)^{-1}=\T^{-1}(u)$. 

(iii) Consider elements
$0\ne x\in C_u$ and $y\in uAu$. If $xy=0$ then, since $x=\T(u)$ for some centralizer $\T$, we have $0=\T(u)y=\T(uy)=\T(y)$ and being $\T$ invertible (because $\clz{(A)}$ is a field) implies $y=0$.
\end{proof}

\begin{lemma}\label{guarrin}
Let $E$ be a graph whose associated Leavitt path algebra $L_\K(E)$ is simple. Then for any vertex $u$ we have $C_u\cap\hbox{Path}(E)=\{u\}=C_u\cap\hbox{Path}(E)^*$.
\end{lemma}
\begin{proof}
First we prove that $C_u\cap\hbox{Path}(E)\subseteq \{u\}$.
Take $\T(u)\in\hbox{Path}(E)$ for some $\T\in \clz(L_\K(E))$. If $\T(u)$ is a trivial path, then, since $\T(u)=u\T(u)u$, we have $\T(u)=u$. If $\T(u)=f_1\cdots f_n$ is a nontrivial path, $s(f_1)=r(f_n)=u$, hence $u$ is the source of the closed path $f_1\cdots f_n$. 

We claim that there is a vertex $w$ in $(f_1\cdots f_n)^0$ such that $s^{-1}(w)$ has at least two elements: indeed, if $f_1\cdots f_n$ turns out to be a cycle, the simplicity of $L_\K(E)$ implies that the cycle has an exit (see \cite[Theorem 3.1.10]{AAS}), hence for some $w$ in $(f_1\cdots f_n)^0$ we have $s^{-1}(w)$ has at least two elements.
If $f_1\cdots f_n$ is not a cycle, then by definition of cycle itself, some $w\in (f_1\cdots f_n)^0$ satisfies 
$\vert s^{-1}(w)\vert\ge 2$. So we may write
$\T(u)=f_1\cdots f_k w f_{k+1}\cdots f_n$ and there is some edge $g\in s^{-1}(w)$, with $g\ne f_{k+1}$.
Then $$(f_1\cdots f_k g g^*f_k^*\cdots f_1^*)\T(u)
=f_1\cdots f_k g g^*f_{k+1}\cdots f_n=0,$$
which contradicts Lemma  \ref{peassomulo} (iii). 

Now suppose that $\T(u)\in\hbox{Path}(E)^*$, say $\T(u)=e_1^*\cdots e_n^*$. Then $r(e_1)=u= s(e_n)$ and $\T(e_n\cdots e_1)=r(e_1)=u$. Using that $L_\K(E)$ is simple, we know that there exists $\T^{-1}$, so $e_n \cdots e_1 = \T^{-1}(u)$ and $e_n \cdots e_1 \in C_u\cap\hbox{Path}(E)$, a contradiction.
\end{proof}

\begin{lemma}\label{spesso}
Let $L_\K(E)$ be the Leavitt path algebra associated to an arbitrary graph $E$,  and let $u,v\in E^0$. Consider a centralizer $\T\in \clz(A)$. Then:
\begin{enumerate}[\rm (i)]
    \item \label{spessoi} If there is a path from $u$ to $v$ and $\T(u)=ku$, where $k\in \K$, then $\T(v)=kv$.
    \item If $u$ is in $\overline{\{v\}}$ and $\T(v)=kv$, where $k\in \K$, then also $\T(u)=ku$.
    \item Assume that $u$ is not in a cycle and $\T(u)\notin \K u$. Let $g$ be an edge which appear in the normal form of $\T(u)$ relative to a basis $B$. If $v=r(g)$, then $\partial_B (\T(u)) > \partial_B (\T(v))$.
\end{enumerate}
\end{lemma}
\begin{proof}
(i) If $\lambda$ is a path with $s(\lambda)=u$ and $r(\lambda)=v$
then $v=\lambda^*\lambda$. Thus, if $\T(u)=k u$ we have $\T(v)=\T(\lambda^*\lambda)=\lambda^*\T(u)\lambda=kv$. 

(ii) For any $w\in T(v)$, by (i) we have $\T(w)=kw$ for $k\in \K$ such that $\T(v)=kv$.
 Now, for $u\in\overline{\{v\}}$, by \cite[Lemma 1.2]{CMMSS} we may write $u=\sum \alpha_i\alpha_i^*$, where each $\alpha_i$ is a path whose range is in $T(v)$. Therefore 
$\T(u)=\sum \alpha_i kr(\alpha_i)\alpha_i^*=k u.$

(iii)
Since $\T(u) \in  Z(uL_\K(E)u)$ (use (i) in Lemma \ref{peassomulo}), by Lemma~\ref{rap} we have a normal expression $\T(u)=ku+\sum f\xi_f f^*$, where  $k\in \K$  and the $\xi_f$'s are elements in $Z\left(r(f)L_\K(E)r(f)\right)$. Then
$g^\ast\T(u)g=kg^\ast ug+\sum g^\ast f\xi_f f^* g$; i.e.,  $\tau(v)=\tau(g^\ast g)= kv + \xi_g$. This implies $\partial_B(\tau(v)) = \partial_B(\xi_g) < \partial_B(\tau(u))$.
\end{proof}

\begin{corollary}\label{porco}
If $L_\K(E)$ is a simple Leavitt path algebra then, for any $u\in E^0$, we have $C_u=\K u$.
\end{corollary}
\begin{proof}
Fix a vertex $u$ and a centralizer $\T$. Applying \cite[ Proposition 3.1 ]{AMMS} to $\T(u)$ and taking into account that $L_\K(E)$ satisfies Condition (L), we get paths $\alpha,\beta$ such that 
$\alpha^*\T(u)\beta=kv\ne 0$ for some vertex $v$ and $k\in \K^\times$.
So $0\ne kv=\T(\alpha^*u\beta)=\T(\alpha^*\beta)$. 
Consequently $0\ne \alpha^*\beta=k\T^{-1}(v)$. Now
$\alpha^*\beta$ is either a path or a ghost path. 
In case $\alpha^*\beta=\gamma$ is a nontrivial path we get a contradiction since $\gamma\in\hbox{Path}(E)\cap C_v=\{v\}$ by
Lemma~\ref{guarrin}. Similarly we get a contradiction if $\alpha^*\beta=\gamma^*$, for a nontrivial path $\gamma$.
Therefore, the only possibility is $\alpha^*\beta=v$ and we conclude that
$\T(v)=kv$. Now, since $L_\K(E)$ is simple, the only nontrivial hereditary and saturated set of vertices is $E^0$, therefore $u\in\overline{\{v\}}$ and hence, by Lemma~\ref{spesso}, we get $\T(u)=ku$.
\end{proof}

\begin{proposition}\label{pasteldenata}
The centroid of a simple Leavitt path algebra $L_\K(E)$ is isomorphic to $\K$.
\end{proposition}
\begin{proof}
Take $\T\in \clz(L_\K(E))$. Since $L_\K(E)$ is simple,  for any vertex $v\in E^0$  we have $\overline{\{v\}}=E^0$ by \cite[Theorem 2.9.1]{AAS}. Fix $v\in E^0$. Corollary \ref{porco} implies that $\tau(v)=kv$ for some $k\in \K$. Now, for any
$u\in \overline{\{v\}}$,  by Lemma~\ref{spesso},  $\tau(u)=ku$ for $k$ as before. Now, apply Proposition \ref{platano} to get that $\tau(x)=kx$ for any $x\in L_\K(E)$.
Thus each centralizer is of the form $k 1_{L_\K(E)}$, for some $k\in \K$  (where $1_{L_\K(E)}$ denotes the identity map).
\end{proof}

\section{Centroids and direct limits}
Recall that a \emph{directed set} $(I,\le)$ is a set with a preorder relation $\le$
such that any two elements $i,j\in I$ have an upper bound, that is, there is an element $k\in I$ such that $i,j\le k$. We will use the notation $k\ge i$ meaning $i\le k$. A \emph{direct system} (or \emph{inductive system}) \emph{of objects} in a category is just a family $\{A_i\}_{i\in I}$ of objects labelled by a directed set $(I,\le)$, and a collection of
arrows $\{e_{ji}\}_{i\le j}$ such that $e_{ji}\colon A_i\to A_j$ satisfying
(1) $e_{ii}=1_{A_i}$ for any $i$, and (2) $e_{kj}e_{ji}=e_{ki}$ when 
$i\le j\le k$. An \emph{inverse} (or \emph{projective}) \emph{system }is a direct system in the opposite category. 
If $S=(\{A_i\}_i,\{e_{ij}\}_{i\le j})$ is a direct system, an object $A$ is a \emph{cocone of} $S$ if there are arrows $e_i\colon A_i\to A$ such $e_je_{ji}=e_i$ whenever $i\le j$. If $A$ and $B$ are cocones of $S$, we have arrows
$e_i\colon A_i\to A$ and $f_i\colon A_i\to B$. Then
an arrow $t\colon A\to B$ is said to be a \emph{homomorphism} from the cocone $A$ to the cocone $B$ if $t e_i=f_i$ for any $i\in I$. 
Similarly, one can define the notion of cone for an inverse systems of objects in a category. If $S$ is a direct system and $A$ is a cocone of $S$, we will say that $A$ is a \emph{direct limit of} $S$ is for any other cocone $B$ of $S$ there is a unique homomorphism of cocones from $A$ to $B$. The notion of inverse or projective limit is dual to this.  

%%%%%%%%%%%%%%%%%%%%%%%%%%%%%%%%%%%%%%%%%%%%%%%%%%%%%%%%%%%%%%%%%%%%%%%%

Let $\K$ be a field. Recall that for an (associative) $\K$-algebra $A$  and any idempotent $e\in A$, we may write  
$$A=e A e\oplus e A f\oplus f A e\oplus f A f,$$ 
where $fA:=\{a-ea \ \vert \ a \in A\}$, $Af:=\{a-ae \ \vert\ a \in A\}$ and $fAf:=\{(a-ea)-(a-ea)e \ \vert \ a \in A\}$. This is called the Peirce decomposition of $A$ relative to the idempotent $e$. If $A$ is unital, then we may take $f=1-e$.
The subspaces $eAe$, $eAf$, $fAe$ and $fAf$ are called the $(1,1)$, $(1,0)$, $(0,1)$ and $(0,0)$ components of the Peirce decomposition of $A$ relative to $e$. Usually the notation for these subspaces is $A_{11}:=eAe$, $A_{10}:=eAf$, $A_{01}:=fAe$ and $A_{00}:=fAf$. Note that $A_{11}$ and $A_{00}$ are subalgebras of $A$. For any Peirce decomposition, there is a $\K$-linear map $\pi\colon A\to A_{11}$ such that $a\mapsto eae$. Of course $\pi$ is not a homomorphism 
of algebras but its restriction $\pi\vert_{A_{11}}$ is a homomorphism $A_{11}\to A_{11}$ (in fact the identity map on $A_{11}$).

Given a field $\K$ and $\K$-algebras $A$ and $B$, with $A$ unital, we will say that $A$ is \emph{nicely embedded in} $B$ if there is a monomorphism $i\colon A\to B$ such that
$i(A)$ coincides with the Peirce $(1,1)$-component of $B$ relative to the idempotent
$i(1_A)$. In this context, we will say that $i$ is a \emph{nice embedding}.
We must remark that when $B$ is unital the monomorphismm $i$ is not necessarilly unital. In case that $i(1_A)=1_B$, where $1_A$ and $1_B$ denote the unital elements in $A$ and $B$, respectively, then the $(1,1)$-component of $B$ relative to $1_B$ is the whole algebra $B$. As a consequence $i$ is an isomorphism. So the definition is interesting specially when $i$ is a monomorphism but does not map the unit of $A$ to the unit of $B$. This happens for instance in the canonical monomorphism
$M_n(\K)\to M_{n+1}(\K)$ such that $\tiny A\mapsto\begin{pmatrix}A & 0\cr 0 & 0\end{pmatrix}$.
\begin{lemma}\label{mulo}
Assume $A$ is nicely embedded in $B$ through a monomorphism $i\colon A\to B$. 
The restriction of $i\colon A\to i(A)$ is an isomorphism of algebras which we will denote $\theta$. Let $e=i(1_A)$ and $\pi\colon B\to A$ the linear map such that $\pi(b)=\theta^{-1}(ebe)$. Then $\pi$ induces  a homomorphism of algebras
$\sigma\colon \clz(B)\to \clz(A)$ such that $\T\mapsto \pi \T i$.
\end{lemma}
\begin{proof}
By definition, $\pi$ restricted to $B_{11}$ is an isomorphism $B_{11}\to A$.
Take $\T\in \clz(B)$ and let $S:=\pi \T i$. We prove first that $S\in \clz(A)$.
For any $x,y\in A$ we have 
$$S(xy)=\pi\left[\T(i(x)i(y))\right]=\pi\left[\T(i(x))i(y)\right]=\pi(\T(i(x))) \pi(i(y)).$$
where this last equality comes from the fact that $\T(i(x))\in B_{11}$ because $\T$ is
a centralizer. Thus $S(xy)=S(x)y$ since $\pi i=1_A$. Symmetrically, we can prove 
$S(xy)=xS(y)$ for any $x,y\in A$. So far, we have $S\in \clz(A)$ and we have a map
$\sigma\colon \clz(B)\to \clz(A)$ such that $\T\mapsto \pi \T i$. Next we prove that $\sigma$ is
a homomorphism of algebras. Take $\T,\T'\in \clz(B)$, then
$\sigma(\T\T')=\pi \T \T' i$. Take now $a\in A$, then 
$$\sigma(\T)\sigma(\tau')(a)=(\pi \T i)(\pi \tau'(i(a)))=(\pi \T i)(\theta^{-1}\tau'(i(a)))$$
and since $i \theta^{-1}(z)=z$ for any $z$ we get 
$$\sigma(\T)\sigma(\tau ')(a)=\pi \T \tau'(i(a))=\sigma(\T\tau')(a).$$
\end{proof}

\begin{lemma}\label{cernicalo}  Assume that $i_1\colon A\to B$ is a nice embedding with $\sigma_1\colon \clz(B)\to \clz(A)$ the induced $\K$-algebras homomorphism according to Lemma \ref{mulo}.
Let $i_2\colon B\to C$ be another nice embedding with $\sigma_2\colon \clz(C)\to \clz(B)$ the corresponding homomorphism between the centroids. Then $i_2i_1$ is a nice embedding with
associated homomorphism $\clz(C)\to \clz(A)$ given by $\sigma_1\sigma_2$.
\end{lemma}
\begin{proof}
We know $i_1(A)=i_1(1_A)Bi_1(1_A)$ and $i_2(B)=i_2(1_B)Ci_2(1_B)$.
From the first equality $i_2i_1(A)=i_2i_1(1_A)i_2(B)i_2i_1(1_A)$ and
so $$i_2i_1(A)=i_2i_1(1_A)(i_2(1_B)Ci_2(1_B))i_2i_1(1_A).$$
But $i_2i_1(1_A) i_2(1_B)= i_2 i_1(1_A)=  i_2(1_B)i_2i_1(1_A)$.
So $$i_2i_1(A)=i_2i_1(1_A)C i_2i_1(1_A).$$
This proves that $i_2i_1$ is a nice embedding. 
Consider now $\sigma_1\colon\clz(B)\to\clz(A)$ such that $\sigma_1(\T)(a)=
\theta_1^{-1}[i_1(1_A)\T(i_1(a))i_1(1_A)]$ where $\theta_1$ is the isomorphism
$\theta_1\colon A\to i(A)$ such that $a\mapsto i_1(a)$.
We also have $\sigma_2\colon \clz(C)\to\clz(B)$ such that 
$\sigma_2(S)(b)=\theta_2^{-1}[i_2(1_B)S(i_2(b))i_2(1_B)]$ being 
$\theta_2\colon B\to i_2(B)$ the isomorphism $b\mapsto i_2(b)$.

Consider now the homomorphism
$\sigma_1\sigma_2\colon\clz(C)\to \clz(A)$ such that for any $S\in\clz(C)$
and $a\in A$, we have $$\sigma_1\sigma_2(S)(a)=(\theta_2\theta_1)^{-1}[i_2i_1(1_A)S(i_2i_1(a))i_2i_1(1_A)]=$$ 
$$\theta_1^{-1}[i_1(1_A)\theta_2^{-1}S(i_2i_1(a))i_1(1_A)=\theta_1^{-1}\theta_2^{-1}S(i_2i_1(a))$$
On the other hand
$$\sigma_1(\sigma_2(S))(a)=\theta_1^{-1}[i_1(1_A)\sigma_2(S)(i_1(a)) i_1(1_A)]=
\theta_1^{-1}[\sigma_2(S)(i_1(a))]=$$
$$\theta_1^{-1}\theta_2^{-1}
[i_2(1_B)S(i_2i_1(a)) i_2(1_B)]=\theta_1^{-1}\theta_2^{-1}
[S(i_2i_1(a))]$$
the last equality coming from the fact that $S(i_2(x))\in i_2(B)=i_2(1_B)Ci_2(B)$.
\end{proof}

Assume that $(I,\le)$ is a directed set and $(A_{i},e_{ji})_{i\le j}$
a direct system of unital $\K$-algebras such that every $e_{ji}\colon A_i\to A_j$ is a nice embedding of $A_i$ in $A_j$. Under this hypothesis we have:
\begin{lemma}\label{naca}
If $\displaystyle A=\lim_\to A_i$, the canonical map $e_i\colon A_i\to A$ is also a nice
embedding.
\end{lemma}
\begin{proof}
Recall that we can take $A=\cup_{i\in I} (A_i\times\{i\})/\equiv$ where the equivalence relation
$\equiv$ is $(x,i)\equiv (y,j)$ if and only if there is some $k$ with 
$i,j\le k$ such that $e_{ki}(x)=e_{kj}(y)$. Then, the induced map $e_i\colon A_i\to A$ is given by $x\mapsto [(x,i)]$ where $[\quad ]$ denotes equivalence class. We prove first that $e_i$ is a monomorphism: if $(x,i)\equiv (0,j)$
there is some $k\in I$ with $i,j\le k$ and $e_{ki}(x)=e_{kj}(0)=0$. This implies $x=0$ since each $e_{ki}$ is a monomorphism. Now we must prove that
$e_i(A_i)=uAu$ with $u=e_i(1_i)$ and $1_i$ the unit of $A_i$. Consider an
arbitrary $[(a,k)]\in A$, then $a\in A_k$ and we may take $k\ge i$. Since all
the maps of the direct systems are nice embeddings we have $e_{ki}(1_i)A_ke_{ki}(1_i)=e_{ki}(A_i)$. So 
$$u[(a,k)]u=[(1_i,i)][(a,k)][(1_i,i)]=[(e_{ki}(1_i),k)][(a,k)][(e_{ki}(1_i),k)]=$$ $$[(e_{ki}(1_i)ae_{ki}(1_i),k)].$$
But $e_{ki}(1_i)ae_{ki}(1_i)\in e_{ki}(1_i)A_ke_{ki}(1_i)=e_{ki}(A_i)$.
So $e_{ki}(1_i)ae_{ki}(1_i)=e_{ki}(z)$ for some $z\in A_i$.
Consequently 
$$[(e_{ki}(1_i)ae_{ki}(1_i),k)]=[(e_{ki}(z),k)]=[(z,i)]\in e_i(A_i).$$
We have proved $uAu\subset e_i(A_i)$. The other relation is trivial.
\end{proof}

Next, we keep on assuming that $(I,\le)$ is a directed set and $(A_{i},e_{ji})_{i\le j}$ a direct system of unital $\K$-algebras such that every $e_{ji}\colon A_i\to A_j$ is a nice embedding of $A_i$ in $A_j$.

\begin{lemma}\label{gorrino}
The induced $\K$-algebra homomorphisms
$\sigma_{ij}\colon \clz(A_j)\to \clz(A_i)$ form an inverse system of algebras and 
\begin{equation}\label{guarro}
\clz(\lim_{\to} A_i)\cong \lim_{\leftarrow} \clz(A_i).
\end{equation}
\end{lemma}
\begin{proof}
As before denote $A=\displaystyle \lim_\to A_i$. By Lemma \ref{naca} the canonical
monomorphisms $e_i\colon A_i\to A$ are nice embeddings so that they induce
algebra homomorphisms $\sigma_i\colon \clz(A)\to \clz(A_i)$. But on the other
hand any $e_{ji}\colon A_i\to A_j$ induces a homomorphism $\sigma_{ji}\colon\clz(A_j)\to\clz(A_i)$ and by Lemma \ref{cernicalo} we have
$\sigma_{ij}\sigma_{jk}=\sigma_{ik}$ when $i\le j\le k$. Let us prove now that
$\displaystyle \lim_\leftarrow\clz(A_i)\cong\clz(A)$. We know that $e_j e_{ij}=e_i$ whenever $i\le j$. This implies (by Lemma \ref{cernicalo}) that 
$\sigma_{ij}\sigma_j=\sigma_i$. Next, we have to prove that for
any $\K$-algebra $U$ and algebra homomorphisms $t_i\colon U\to\clz(A_i)$
satisfying $\sigma_{ji}t_i=t_j$ for $j\le i$, there is a unique algebra homomorphism $t\colon U\to\clz(A)$ such that $\sigma_i t=t_i$. So, in order to define $t$ take an arbitrary $u\in U$. Let us prove the commutativity of the diagrams:
\[
\xygraph{
!{<0cm,0cm>;<1cm,0cm>:<0cm,1cm>::}
!{(0,0)}*+{A_j}="a"
!{(1.5,0)}*+{A_i}="b"
!{(1.5,-1.3)}*+{A}="c"
"a":^{e_{ij}}"b"
"b":^{e_i t_i(u)}"c"
"a":_{e_jt_j(u)}"c"
}\]
when $j\le i$. We have to prove
$e_i t_i(u)e_{ij}=e_jt_j(u)$ and we
know that $t_j(u)=\pi_{ji}t_i(u)e_{ij}$ where 
$\pi_{ji}\colon A_i\to A_j$ satisfies $\pi_{ji}e_{ij}=1_{A_j}$. So 
$e_j t_j(u)=e_j\pi_{ji}t_i(u)e_{ij}$. But for any $x\in A_j$
$$e_jt_j(u)=e_j\pi_{ji}t_i(u)e_{ij}(x)=e_j\pi_{ji}e_{ij}(z)=e_j(z)$$
where $t_i(u)(e_{ij}(x))=e_{ij}(z)$ for some $z$ (because $t_i(u)$ is a centralizer and $e_{ij}$ a nice embedding). On the other hand
$$e_it_i(u)e_{ij}(x)=e_ie_{ij}(z)=e_j(z)$$
whence $e_jt_j(u)=e_it_i(u)e_{ij}$. Then, by the universal property of direct limits (taking into account that the underlying vector space of $A$ is the limit in the category of vector spaces of the  direct system of underlying vector spaces), there is a unique linear map $t(u)\colon A\to A$ such that 
$t(u)e_i=e_it_i(u)$ for any $i$. Next we prove that $t(u)\in\clz(A)$: taking $x,y\in A$ we know that there is some $i\in I$ such that 
$x,y\in \hbox{Im}(e_i)$. Thus $x=e_i(x')$ and $y=e_i(y')$ and so
$$t(u)(xy)=t(u)e_i(x'y')=e_i t_i(u) (x'y')=e_i[t_i(u)(x')y']=$$
$$e_i(t_i(u)(x'))e_i(y')=t(u)(e_i(x')) y=(t(u)x)y$$ and similarly
$t(u)(xy)=x t(u)(y)$ for any $x,y\in A$. Consequently $t(u)\in\clz(A)$.
Finally we prove that $\sigma_i t=t_i$ for any $i$. Since $e_i\colon A_i\to A$ is a nice embedding we have the Peirce decomposition of $A$ relative to $e_i(1)$, that is, $A=A_{11}\oplus A_{10}\oplus A_{01}\oplus A_{00}$ with $A_{11}=e_i(A_i)$.
Then $\pi_i\colon A\to A_i$ is the canonical epimorphism.  
If we take an $u\in U$,
then $\sigma_i(t(u))=\pi_i t(u) e_i=\pi_i e_i t_i(u)=t_i(u)$. 
This finishes the proof that $\clz(A)\cong\displaystyle \lim_{\leftarrow}\clz(A_i)$.
\end{proof}

Let us give now two examples that may illustrate the use of the formula (\ref{guarro}). Recall from \cite[Definition 3.2]{AAS1} that a row-finite graph $E$ is called a {\it comet} if it has exactly one cyle $c$, $T(v)\cap c^0\neq \emptyset$ for every $v\in E^0$, and every infinite path ends in the cycle $c$.

%\textcolor{red}{ Daniel: 
%We have to make it clear that comet is defined for row-finite graphs. If we allow for non row finite graphs then we would have problems. I added row-finite in the proposition below, recalled definition comet above, and added a remark after the proof. }
 
 For instance the Leavitt path algebra $A$ associated to the graph:

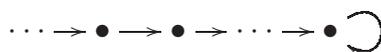
\[
\xygraph{
!{<0cm,0cm>;<1cm,0cm>:<0cm,1cm>::}
!{(-1,0)}*+{\cdots}="a"
!{(0,0)}*+{\bullet}="b"
!{(1,0)}*+{\bullet}="c"
!{(2,0)}*+{\cdots}="d"
!{(3,0)}*+{\bullet}="e"
"a":"b" "b":"c" "c":"d"
"d":"e" "e" :@(ru,rd) "e"
}\]

\noindent This algebra is not simple but it is prime. To compute its centroid, we know 
$$A\cong M_\infty(\K[x,x^{-1}])\cong\lim_{\to}M_n(\K[x,x^{-1}])$$
and $\clz(M_n(\K[x,x^{-1}])\cong \K[x,x^{-1}]$. So, by Lemma \ref{gorrino} we have

\begin{equation}\label{garrulo}
    \clz(A)=\lim_{\leftarrow} \clz(M_n(\K[x,x^{-1}])=
\lim_{\leftarrow}\K[x,x^{-1}]=\K[x,x^{-1}].
\end{equation}

Let us try with the centroid of $A=L_\K(E)$ when $E$ is the graph:
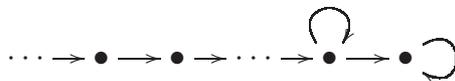
\[
\xygraph{
!{<0cm,0cm>;<1cm,0cm>:<0cm,1cm>::}
!{(-1,0)}*+{\cdots}="a"
!{(0,0)}*+{\bullet}="b"
!{(1,0)}*+{\bullet}="c"
!{(2,0)}*+{\cdots}="d"
!{(3,0)}*+{\bullet}="g"
!{(4,0)}*+{\bullet}="e"
"a":"b" "b":"c" "c":"d"
"d":"g" "e" :@(ru,rd) "e"
"g":"e"
"g" :@(ul,ur) "g"
}\]
This algebra is prime but not simple.
Roughly speaking $A=\lim_\to A_n$ where $A_n$ is the Leavitt path algebra of the finite graph $E_n$:
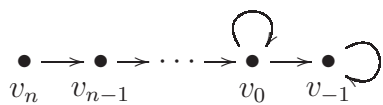
\[
\xygraph{
!{<0cm,0cm>;<1cm,0cm>:<0cm,1cm>::}
!{(0,0)}*+{\bullet}="b"
!{(0,-0.4)}*+{v_n}
!{(1,0)}*+{\bullet}="c"
!{(1,-0.4)}*+{v_{n-1}}
!{(2,0)}*+{\cdots}="d"
!{(3,0)}*+{\bullet}="g"
!{(3,-0.4)}*+{v_0}
!{(4,0)}*+{\bullet}="e"
!{(4,-0.4)}*+{v_{\tiny -1}}
"b":"c" "c":"d"
"d":"g" "e" :@(ru,rd) "e"
"g":"e"
"g" :@(ul,ur) "g"
}\]
and by the finiteness of $E_n$ we have $\clz(A_n)=Z(A_n)=\K 1$. So
$$\clz(A)=\lim_{\leftarrow}\clz(A_n)\cong \lim_{\leftarrow} \K\cong \K.$$ 

Consider finally the graph 
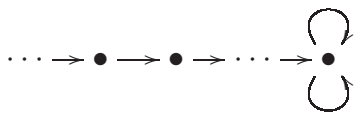
\[
\xygraph{
!{<0cm,0cm>;<1cm,0cm>:<0cm,1cm>::}
!{(-1,0)}*+{\cdots}="a"
!{(0,0)}*+{\bullet}="b"
!{(1,0)}*+{\bullet}="c"
!{(2,0)}*+{\cdots}="d"
!{(3,0)}*+{\bullet}="g"
%!{(4,0)}*+{\bullet}="e"
"a":"b" "b":"c" "c":"d"
"d":"g" 
"g" :@(ul,ur) "g"
"g" :@(dl,dr) "g"
}\]
again with an infinite \lq\lq tail\rq\rq\ .
Its Leavitt path algebra $A$ is the direct limit of the Leavitt path algebras of finite graphs:
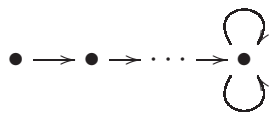
\[
\xygraph{
!{<0cm,0cm>;<1cm,0cm>:<0cm,1cm>::}
%!{(-1,0)}*+{\cdots}="a"
!{(0,0)}*+{\bullet}="b"
!{(1,0)}*+{\bullet}="c"
!{(2,0)}*+{\cdots}="d"
!{(3,0)}*+{\bullet}="g"
%!{(4,0)}*+{\bullet}="e"
%"a":"b" 
"b":"c" "c":"d"
"d":"g" 
"g" :@(ul,ur) "g"
"g" :@(dl,dr) "g"
}\]
with a \lq\lq finite\rq\rq\ tail. It is not difficult to realize that the centroid of these algebras (which agrees with their centers) is isomorphic to $\K$. Hence the centroid $\clz(A)$ is the inverse limit of a projective system in which all the algebras are $\K$. Thus $\clz(A)\cong \K$ again.

\begin{proposition}\label{gorrona}
Let $A=L_{\K}(E)$ be the Leavitt path algebra associated to a row-finite graph $E$ wich is a comet. Then $\clz(A)$ is isomorphic to $\K[x,x^{-1}]$.
\end{proposition}
\begin{proof}
By \cite[Proposition 3.4]{AAS}
we have $L_{\K}(E)\cong M_{\Lambda}(\K[x,x^{-1}])$ wiht $\Lambda $ finite or infinite. In the first case,  formula \eqref{ugna} gives that $\clz(A)=Z(A)=\K[x,x^{-1}]$.
In the second case, applying formula \eqref{garrulo} we have $\clz(A)=\K[x,x^{-1}]$.
\end{proof}

To finish this section we analyze the centroid of a graded simple, non-simple, Leavitt path algebra.

\begin{proposition}\label{PamDeQueijo}
Let $L_\K(E)$ be a row-finite, graded simple, non-simple, Leavitt path algebra. Then $E$ is a comet with $L_\K(E)\cong M_{\Lambda}(\K[x,x^{-1}])$ and its centroid is isomorphic to $\K[x, x^{-1}]$. 
%Furthermore, the extended centroid of $L_\K(E)$ is isomorphic to $\K(x)$, the field of rational functions on $x$.
\end{proposition}
\begin{proof}
By \cite[Corollary 2.5.15]{AAS}, the only hereditary and saturated subsets of $E^0$ are $\emptyset$ and $E^0$. Since $L_\K(E)$ is graded simple but not simple, there exists a  cycle without exits. There is only one cycle without exits because, otherwise, the cardinal of $\mathcal{H}_E$ would be strictly greater than 2, a contradiction. Denote by $c$ this unique cycle. We claim that there are no more cycles. Assume, on the contrary, that $d$ is another cycle, necessarily with an exit. Following the same notation as in \cite[Definition 2.0.6]{AAS}, for $X=c^0$, we consider the sets $X_n$.
We prove that $d^0\cap X_n=\emptyset$ for any $n>=0$. 
Let $u\in d^0$. This vertex cannot be in $X_0$ because $d^0\neq c^0$. Assume now $d^0\cap X_n=\emptyset$ and prove that $d^0\cap X_{n+1}=\emptyset$. If this is not the case, there is an $u\in d^0$ such that $u\in X_{n+1}$ hence $u$ is a regular vertex and $r(s^{-1}(u))\in X_n$. But for some $f\in s^{-1}(u)$
we have $r(f)\in d^0$ (and $r(f)\in X_n$ a contradiction).
So, there is only a cycle and any vertex connects with the cycle. 
To conclude that $E$ is a comet we need to 
prove that any infinite path ends in the cycle (see \cite[Definition 3.2]{AAS1}): indeed,
let $\l$ be an infinite path, since $\overline{\l^0}=E^0$ we have 
$\overline{\l^0}\cap c^0\ne\emptyset$ so that there is some 
$v\in c^0\in \overline{\l^0}$ which by \cite[Lemma 1.2]{CMMSS} gives that 
$v$ connects with $\l^0$ hence some vertex of $\l$ is in $c$ which implies
that $\l$ \lq\lq ends\rq\rq\ in the cycle. 
Then $E$ is a comet and 
applying \cite[Theorem 2.7.3]{AAS} we have an isomorphism of $L_\K(E)$ with a direct sum whose summands are of the type  $M_{\Lambda}(\K[x,x^{-1}])$ (take into account that, since the graph is row finite, in our case the ideal generated by vertices in cycles without exits is the whole algebra). Now, the primeness of $L_\K(E)$ (by \cite[Proposition II.1.4]{NO}, a graded $\Z$-algebra is graded prime if and only if it is prime)  implies that $L_\K(E)\cong M_\Lambda(\K[x,x^{-1}])$ for some possible infinite set $\L$. The fact that $\clz(L_\K(E))\cong \K[x,x^{-1}]$ is given in formula (\ref{garrulo}).

%To end, we compute the extended centroid of $L_\K(E)$. Since $L_\K(E)$ is isomorphic to $M_\Lambda(\K[x,x^{-1}])$, then its extended centroid is the center of the maximal algebra of left quotients of $M_\Lambda(\K[x,x^{-1}])$, i.e., the center of $CFM_\Lambda(\K(x))$, which is isomorphic to $\K(x)$, where  $CFM_\Lambda(R)$ are column finite matrices over $R$.

%\textcolor{magenta}{\large\underline{Pendiente de \textcolor{red}{que Mercedes} busque una referencia.}}
\end{proof}

In \cite{AAS1} the authors consider only row finite graphs and so comets were defined in this context. But the definition of a comet can be read for arbitrary graphs without change. In this more general setting we can not use the result above to compute the centroid of the algebra. For example, let $E$ be a graph with  two vertices, $u$ and $v$, where $u$ is an infinite emitter such that the range of each edge in $s^{-1}(u)$ is $v$, and $v$ is the base of cycle without exit. This is a (non row-finite) comet, but $L_\K(E)$ is not isomorphic to $ M_{\Lambda}(\K[x,x^{-1}])$. We will compute the centroid of this example in the next section.

\begin{remark}\label{repetido}

It is also possible to prove directly (without the use of direct limits) that $\clz(\LPA) \cong \K[x,x^{-1}]$ for $\LPA$ graded simple but non-simple. To see this, let $c$ be the unique cycle without exits  in $E$. Let $u\in c^0$ and define a map $\Omega: \clz(\LPA) \rightarrow C_u $ by $\Omega (\tau) = \tau (u)$. It is clear that $\Omega$ is a surjective homomorphism. To see that $\Omega$ is injective assume that $\tau(u)=0$. By Lemma~\ref{spesso}[i] $\tau$ vanishes on $T(u)$. 
We know that $\overline{\{u\}}=E^0=\cup_n\Lambda^n$ (for $X=\{u\}$ and following the terminology in section 2). Suppose by induction that $\tau$ vanishes at $\Lambda^n$. Let $v \in \Lambda^{n+1} \setminus \Lambda^n$. Then $v= \sum f_i f_i^*$, where $r(f_i) \in \Lambda^n$. Hence $\tau(v) = \sum f_i \tau (r(f_i)) f_i^* = 0$ and $\tau(\overline{\{u\}})=0$. Since $\overline{\{u\}} = E^0$ we obtain that $\tau = 0$. Therefore $ \clz(\LPA)$ is isomorphic to $C_u$. 
To see that $C_u$ is isomorphic to $\K[x,x^{-1}]$ define $\psi: C_u \rightarrow \K[x,x^{-1}]$ in the following way: given $\tau(u)\in C_u$, since $C_u \subseteq u \LPA u = \K[c, c^*]$, then $\tau(u)= p(c,c^*)$ for some polynomial $p\in \K[x,x^{-1}]$. Define $\psi(\tau(u)) = p(x, x^{-1})$. It is clear that $\psi$ is an injective homomorphism. To see that $\psi$ is surjective, let $p\in \K[x,x^{-1}]$. Define $\tau(u):= p(c,c^*)$. For any $v\in \Lambda^0=T(u) = c^0$, write $c=\sigma \sigma'$, where $s(\sigma')=v$. Let $\tau(v):=\sigma^*p(c,c^*) \sigma$.
By induction, suppose we have defined $\tau$ in $\Lambda^n$. Let $w\in \Lambda^{n+1}\setminus \Lambda^n$. Then $s^{-1}(w)=\{g_1,\ldots,g_k\}$, where $r(g_i)\in \Lambda^n$ for all $i$. Define $\tau(w):=\sum_{i=1}^k g_i \tau(r(g_i)) g_i^*$. Now, by Proposition~\ref{platano}, $\tau$ is uniquely extended to $\clz( \LPA)$ and we are done. 
\end{remark}

\section{The prime case}
As the title says in this section we will deal with prime Leavitt path algebras. Recall that a Leavitt path algebra $L_\K(E)$ is prime if and only if the graph satisfyes Condition~(MT3), which is also known as downward directedness (see \cite{Gene}, \cite{BLR}).
If $L_\K(E)$ is  prime, we will observe some cases in which the centroid is $\K$ and others in which it is $\K[x,x^{-1}]$. It is known that in general, the centroid of a prime algebra is a domain.

The scheme in Figure \ref{rfvfr} below explains the tree dichotomies that we have followed to consider all the possible cases. Observe for instance that if $E$ satifies MT3 and has cycles, the contrary predicate of \lq\lq $\exists$ cycle with exits\rq\rq\ is \lq\lq $\exists !$ cycle without exits\rq\rq, where the symbol $\exists !$ stands for \lq\lq exists a unique\rq\rq. Note that, under the  previous conditions, the predicate \lq\lq  $\exists !$ cycle without exits\rq\rq\ is equivalent to the assertion that there is a unique cycle and it has no exits.

\begin{center}
\begin{figure}[H]
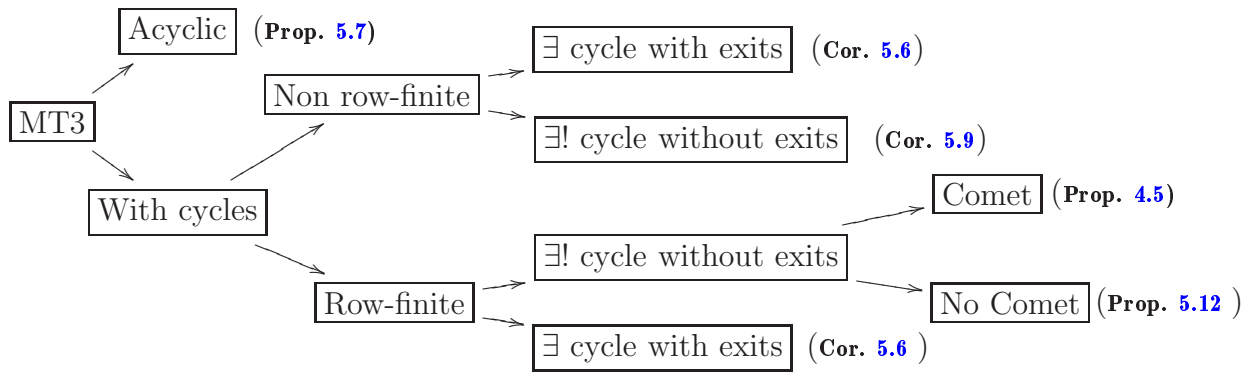

\xygraph{!{<0cm,0cm>;<1.5cm,0cm>:<0cm,1.2cm>::}
!{(0,0)}*+{{\fbox{MT3}}}="a"
!{(1.1,1)}*+{\fbox{Acyclic}}="b"
!{(2.3,1)}*+{(\hbox{\s Prop. \ref{sol})}}
!{(7.15,0.8)}*+{(\hbox{\s Cor. \ref{Billy}})}
!{(7.7,-0.2)}*+{(\hbox{\s Cor. \ref{infinite}})}
!{(9.3,-0.8)}*+{(\hbox{\s Prop. \ref{gorrona})}}
!{(9.7,-2)}*+{\ \ (\hbox{\s  Prop. \ref{kfield} })}
!{(7.15,-2.5)}*+{(\hbox{\s Cor. \ref{Billy} })}
!{(1.1,-1)}*+{\fbox{With cycles}}="c"
!{(2.8,0.3)}*+{\fbox{Non row-finite}}="d"
!{(5.35,0.8)}*+{\fbox{$\exists$ cycle with exits}}="i"
!{(5.6,-0.2)}*+{\fbox{$\exists !$ cycle without exits}}="j"
!{(3,-2)}*+{\fbox{Row-finite}}="e"
!{(5.6,-1.5)}*+{\fbox{$\exists !$ cycle without exits}}="f"
!{(8.2,-0.8)}*+{\fbox{Comet}}="k"
!{(8.4,-2)}*+{\fbox{No Comet}}="l"
!{(5.35,-2.5)}*+{\fbox{$\exists$ cycle with exits}}="g"
"a":"b"
"a":"c"
"c":"d"
"c":"e"
"e":"f"
"e":"g"
"d":"i"
"d":"j"
"f":"k"
"f":"l"
%"i":"j"
%"i":"k"
%"l"-"m"
%"n"-"p"
}
\medskip
\caption{Decision tree}\label{rfvfr}
\end{figure}
\end{center}

Before we proceed analysing each case we prove a few general auxiliary results.

\begin{lemma}\label{Charlie}
Let $A$ be a prime algebra and $u$ an idempotent in $A$. Then:
\begin{enumerate}
\item For any nonzero $\T\in\clz(A)$ we have that $\T$ is a monomorphism.
\item For any $0\ne x\in C_u$ and $0\ne y\in uAu$ we have $xy\ne 0$.
\item Let $A=L_\K(E)$ and $0\ne \T\in\clz(A)$. If $\gamma\in\path(E)$ and $\T(\gamma)=ku\ne 0$ with $u\in E^0$, then $\gamma\in Z(uAu)$.
\item If $\gamma\in Z(uAu)\cap\path(E)$ for a Leavitt path algebra $A=\LPA$ and $u$ is the base of a cycle with exits, then $\gamma$ is trivial. 
\end{enumerate}
\end{lemma}
\begin{proof}
For the first assertion we know that $\ker(\T)$ and $\hbox{im}(\T)$ are ideals of $A$ and 
$\ker(\T)\hbox{im}(\T)=0$ (indeed, if $\T(x)=0$ then $x \T(y)=\T(x)y=0$).
So by primeness of $A$ we have $\ker(\T)=0$ (because $\T\ne 0$).
For the second assertion assume that $x$ and $y$ are nonzero elements in $C_u$ and $uAu$ respectively. So $x=\T(u)$ for some $\T\in\clz(A)$ and 
$xy=\T(u)y=\T(y)$. If $xy=0$ we deduce that $\T(y)=0$ and so $y=0$ a contradiction. To prove the third assertion
take an arbitrary $z\in uAu$, then $\T(\gamma)z=kuz=kz=z\T(\gamma)$. So $\T(\gamma z)=
\T(z\gamma)$ and since $\T$ is a monomorphism $\gamma z=z\gamma$ whence $\gamma\in Z(uAu)$. For the fourth item assume that $\gamma\ne u$, say $\gamma = g_1 \ldots g_m $, and let $c$ be the cycle with $s(c)=r(c)=u$. Write $c=\sigma \sigma'$ (with $\sigma'$ nontrivial) so that $v:=s(\sigma ')$ is an exit for $c$. Then, since $\gamma$ and $c:=c_1 \ldots c_n$ commute, we get that $$g_1 \ldots g_m c_1 \ldots c_n = c_1 \ldots c_n g_1 \ldots g_m.$$
Since %$\path(E)$ is a linearly independent set 
%\textcolor{blue}{(Is it neccessary the linear independence???)}
%in $\LPA$, and 
$c$ is a cycle, we get that $\gamma = c \beta$ for some path $\beta$. Now, let $f$ be an edge such that $s(f)=v$ and $f$ is different from the first edge of $\sigma '$. Then,
since $\gamma \in Z(uAu)$, we obtain that $$\gamma (\sigma f f^* \sigma^*)=(\sigma f f^* \sigma^*) \gamma = (\sigma f f^* \sigma^*) c \beta = (\sigma f f^* \sigma^*) \sigma \sigma' \beta = \sigma f f^*  \sigma' \beta =0.$$ 
Hence, multiplying on the right by $\sigma f$, we have that $\gamma \sigma f = 0$, a contradiction.
\end{proof}

The idea on Remark~\ref{repetido} of identifying $\clz(A)\to C_u$, for some $u\in E^0$, is key in the sequel. So we make a precise statement for prime algebras below.

\begin{proposition}\label{Paul}
For a prime Leavitt path algebra $A=L_\K(E)$ and any $u\in E^0$, the map $\Omega\colon\clz(A)\to C_u$ such that $\Omega(\tau):=\tau(u)$ is an isomorphism.
\end{proposition}
\begin{proof}

Let $\tau,\sigma\in\clz(A)$. Then $$\Omega(\tau\sigma)=\tau(\sigma(u^2))=
\tau(u\sigma(u))=\tau(u)\sigma(u)=\Omega(\tau)\Omega(\sigma).$$ Also by construction $\Omega$ is surjective and by Lemma \ref{Charlie} it is a monomorphism. 
\end{proof}

%So to determine the centroid of $A$ all we need is to determine $C_u$ at any vertex $u\in E^0$, which is an interesting reduction.

Given the above proposition our next goal is to identify $\tau (u)$ when $u$ is the base of a cycle and $\tau$ is a centralizer. For this we need the two auxiliary lemmas below.

\begin{lemma}\label{vaiBrasil: Essa copa eh nossa!}
Let $c$ be a cycle of $\LPA$ based at $u$ and consider the map $S\colon uAu\to uAu$ given by $S(x)=c^*xc$. Assume that
$w\in uAu$ is such that $S^n(w)\ne 0$ for each $n\ge 1$ and that $S(w)\in \path(E)\cup\path(E)^*$. Then $S(w)=c^m$ for some $m\in\Z$.
\end{lemma}
\begin{proof}
Assume first that $\lambda:=S(w)\in\path(E)$. Then $c^*\lambda c\ne 0$ implies the following dichotomy:
\begin{enumerate}
\item There is a maximum natural $n\ge 1$ and a path $\mu$ with $\lambda=c^n \mu$.
\item There is a maximum natural $n\ge 1$ and a path $\mu$ with $c=\lambda^n\mu$.
\end{enumerate}
In the first possibility, if $\mu=u$ we have $\lambda=c^n$ and we are done. Thus we may assume $\mu\ne u$. Then  $c^*\mu=0$ since otherwise
$c=\mu\tau$ for some $\tau\in\path(E)$. But since $\mu,\tau\in uAu$ and $c$ is a cycle we have
$$\begin{cases} \mu=c, \tau=u \cr \text{ or }\cr
\mu=u, \tau=c\end{cases},$$
however both possibilities above have been already excluded.
Thus $c^*\mu=0$ and then $$S^{n+2}(w)={c^*}^{n+1}\lambda c^{n+1}={c^*}^{n+1} c^n\mu c^{n+1}=c^*\mu c^{n+1}=0,$$
a contradiction.

The second possibility of the dichotomy is that $c=\lambda^n\mu$ for a maximum $n$ and certain paths $\lambda,\mu\in uAu$. Since $c$ is a cycle and $n\ge 1$, we have $n=1$ and either $c=\lambda$ (in which case we are done) or $\lambda=u$ and we are also done.

Finally if $S(w)=\l^*$ with $\l\in\path(E)$,
we have $S(w^*)=S(w)^*=\l$ and, applying the previous discussion, we get again $S(w)=c^m$ for some integer $m$ (the powers of negative exponent as usual are powers of $c^*$ with positive exponent).
\end{proof}

\begin{lemma}\label{messinoestabien}
Let $c$ be a cycle of $\LPA$ based at $u$, $w=\alpha \beta^*$ be a walk in $u\LPA u$ and $S$ as in Lemma~\ref{vaiBrasil: Essa copa eh nossa!}. Suppose that $S^n(w)\neq 0$ for all $n\geq 1$. Then there exists an $m\in \N$ such that $S^m(w)$ belongs to $\text{Path}(E)\cup \text{Path}(E)^*$.
\end{lemma}
\begin{proof}
Let $w=\a \b^* \in u\LPA u$. If $w=u$ the result follows directly. Write $\a = c^k \a' $ and $\b = c^q \b'$ where $k$ and $q$ are non negative integers and $\a',\b'$ are paths such that $\a' \neq c \a''$ and $\b' \neq c \b''$, for all $\a'', \b'' \in \text{Path}(E)$. We have the following possibilities:
\begin{itemize}
    \item If $q>k$, then $S^k(w)=\a'\b'^* (c^{q-k})^*$ and $S^{k+1}(w)=c^*\a'\b'^*(c^{q-k-1})^*\ne 0$. Thus $c=\a'\mu$ for some path $\mu$ and consequently $S^{k+1}(w)=\mu^*(\b')^*(c^*)^{q-k-1}$ is a ghost path.
    \item If $q<k$, reasoning on $S^q(w)$ and 
    $S^{q+1}(w)$, we get that the latter is a real path.
    \item If $q=k=0$,  we have that $\a$ and $\b$ are not multiples of $c$. Since $0\ne S(w)$ we have $c=\a\mu=\b\lambda$ for some paths $\mu,\lambda$. But then $S(w)=\mu^*\a^*\a\b^*\b\lambda=\mu^*\lambda$ and this is a real or a ghost path (being nonzero).
    \item If $q=k>0$, then $w=c^k\a'\b'^*(c^*)^k$ and $S^k(w)=\a'\b'^*$ which proves that 
    $S^n(\a'\b'^*)\ne 0$ for any $n$. By  the previous item, applied to $\a'\b'^*$, we know that for some integer $m$, we have 
    $S^m(\a'\b'^*)$ is either a path or a ghost path. But $S^{k+m}(w)=S^m(\a'\b'^*)$.
\end{itemize}
\end{proof}

\begin{proposition}\label{Desmond}
Let $u\in E^0$ be the base of a cycle $c$ and
$\tau\in\clz(\LPA)$. Then $\tau(u)$ is a Laurent polynomial in $c$.
\end{proposition}

\begin{proof}
Observe first that $\tau(u)\in Z(uAu)$ and hence
$\tau(u)\in\hbox{Fix}(S):=\{x\in\LPA\colon S(x)=x\}$. Of course $\tau(u)\in\hbox{Fix}(S^m)$ for any $m\ge 1$.
We write $\tau(u)=k_1w_1+\cdots+k_nw_n$ where $k_i\in\K^\times$, the $w_i$'s are walks and $n$ is minimum. Then for any $m$ we have $\tau(u)=S^m(\tau(u))=\sum_{i=1}^n k_i S^m(w_i)$ and hence
$S^m(w_i)\ne 0$ for each $i=1,\ldots,n$ (and for arbitrary $m$). Applying Lemma~\ref{messinoestabien} we get for each $i$ the existence of an exponent $q_i$ such that 
$S^{q_i}(w_i)\in\path(E)\cup\path(E)^*$. Thus,
taking $t\ge\max(q_i)$ we have that $\tau(u)=S^t(\tau(u))=\sum_i k_i S^t(w_i) $ is
a linear combination of path or ghost paths.
Applying now Lemma~\ref{vaiBrasil: Essa copa eh nossa!} we get that each $\tau(u)$ is a linear combination of  powers (possibly negative) of $c$.
\end{proof}

The above proposition allow us to identify the centroid of prime Leavitt path algebras associated to graphs that posses a cycle with exit. 

\begin{corollary}\label{Billy}
Under the hypothesis of the previous proposition if the cycle has an exit, then 
$\tau(u)\in\K u$. In particular, if $A=L_\K(E)$ is prime and there is a cycle with exits in $E$, we have 
$\clz(A)\cong\K$.
\end{corollary}
\begin{proof}
By Proposition \ref{Desmond}, $\tau(u)=\sum_i k_i c^i$ a polynomial in $c$.
If $c$ has an exit we may write $c=\sigma\lambda$ with 
$\sigma,\lambda \in\path(E)$ and $\lambda$ nontrivial, in such a way that there is an edge $f$ which is an exit for $c$ and $s(f)=s(\lambda)$, but $f$ does not coincide with the first edge of $\lambda$, that is, $f^*\lambda=0$. Then by (i) of Lemma \ref{peassomulo}, imposing commutativity of 
$\sum_{i}k_i c^i$ with $\sigma f f^*\sigma^*$, we get $k_i=0$ for $i\ne 0$, that is, $\tau(u)\in\K u$.
Indeed: we can write $\tau(u)=\sum_{i}k_i c^i=k_0 u+p+q$ where $p$ is a polynomial in $c$ and $q$ a polynomial in $c^*$ both of positive degree. Observe that $q \sigma f f^*\sigma^*=0=\sigma f f^*\sigma^* p$ and consequently the commutativity of $\tau(u)$ and $\sigma f f^*\sigma^*$, equating terms of the same degree,  gives $k_i=0$ for $i\ne 0$.  For the second part of the Corollary consider a cycle with exits $c$ and apply Proposition~\ref{Paul}.
Then $\clz(A)\cong C_u=\K u\cong\K$.
\end{proof}

After Corollary \ref{Billy}, we must focus our attention of prime Leavitt path algebras associated to graphs $E$ in which every cycle (if any) is a no-exit cycle. If $L_\K(E)$ is prime and there is a cycle with no exits, there is only one such a cycle.  In the case of absence of cycles, or the existence of an infinite emitter, the centroid is isomorphic to $\K$, as we show below. 

\begin{proposition}\label{sol}
Consider a prime Leavitt path algebra $A=L_\K(E)$  with $E$ acyclic. Then $\clz(A)\cong\K$.
\end{proposition}
\begin{proof}
By Lemma \ref{peassomulo}, if $u \in E^0$, then $\tau(u)\in \Z(uAu)$. By Proposition \ref{criacaballar} we have $\tau(u)=ku$ for some $k \in \K$ and by Proposition \ref{Paul} we conclude $C_u\cong \clz(A)$, hence $\clz(A)\cong\K$.
%If $E$ has an infinite emitter $u\in E^0$, we know that $\clz(A)\cong C_u$ and by Corollary \ref{criacaballar}  $C_u=\K u$. Otherwise $E$ is row-finite and then we apply Corollary \ref{criacaballar}  so that again $\clz(A)\cong\K$.
\end{proof}

\begin{proposition}\label{infinite1}

Let $E$ be a graph satisfying MT3 and suppose that there is 
  an infinite emitter $v\in E^0$ which is not the base of a cycle. If $z\in Z(vL_{\K}(E)v)$ then  $z\in Kv$.
 \end{proposition}

 \begin{proof}
If we write $z$ in normal form relative to a basis $B$, then $z=kv+\sum_{i=1}^n k_i\a_i\b_i^*$ where $k, k_i\in K$ and the $\a_i$'s and the $\b_i$'s are real
paths of length $\ge 1$ whose source is $v$ and $r(\alpha_i)=r(\b_i)$ (recall that $v$ is not base of a cycle). Write also $s^{-1}(v)=\{f_j\}_{j\in J}$. Since $s^{-1}(v)$ is
infinite we can select $f_j\in s^{-1}(v)$ such that $\b_i^*f_j=0$ for  $i=1,\ldots,n$. 

Assume that some $k_i$ in the expression of $z$ is not zero. For the nonzero scalars $k_i$
select one of the $\a_i$'s of maximal length.
So we fix $i_0$ such that $\hbox{length}(\a_{i_0})\ge\hbox{length}(\a_i)$ for any $i$. Then $k_{i_0}\ne 0$ (this is important because we will get a contradiction to this in due course).
Since the graph satisfies MT3 there are paths $\lambda$ and $\mu$ such that
$s(\lambda)=r(f_j)$, $s(\mu)=r(\a_{i_0})$ and $r(\lambda)=r(\mu)$. So $\eta:=f_j\lambda\mu^*\a_{i_0}^*\in vLv$ hence $z$ commutes with $\eta$. But $z\eta=k\eta+\sum_i k_i\a_i\b_i^*f_j\lambda\mu^*\a_{i_0}^*=k\eta $ since $\b_i^*f_j=0$ for any $i$. On the other hand, 
$\eta z=k\eta+\sum_i k_i f_j\lambda\mu^*\a_{i_0}^*\a_i\b_i^*$ and given that the length of each $\a_i$ is less than or equal to the length of $\a_{i_0}$, we can write $\a_{i_0}=\a_{i}\gamma_i$ (otherwise $\a_{i_0}^*\a_i=0$). So 
$\eta z=k\eta +\sum_i k_i f_j\lambda\mu^*(\a_{i}\gamma_i)^*\a_i\b_i^*=k\eta +\sum_i k_i f_j\lambda\mu^*\gamma_i^*\b_i^*$. Consequently  $\sum_i k_i f_j\lambda\mu^*\gamma_i^*\b_i^*=0$ and observe that there is at least one 
nonzero $f_j\lambda\mu^*\gamma_i^*\b_i^*$, precisely  $f_j\lambda\mu^*\gamma_{i_0}^*\b_{i_0}^*$ (because if this element is zero then $f_j\lambda=0$ which is a contradicton). So 
$$0=\sum_i k_i f_j\lambda\mu^*\gamma_i^*\b_i^*=(f_j\lambda\mu^*)\sum_i k_i \gamma_i^*\b_i^*$$ 
thus $0=\lambda^* f_j^* (f_j\lambda\mu^*)\sum_i k_i \gamma_i^*\b_i^*= \sum_i k_i \mu^*\gamma_i^*\b_i^*$ implying $k_i=0$
because any collections of real or ghost paths is linearly independent. So far we have proved that 
$k_i=0$ if $\a_{i_0}=\a_i\gamma_i$, but this is the case precisely for $i_0$. Thus $k_{i_0}=0$ a contradiction.
\end{proof}

\begin{corollary}\label{infinite} Under the conditions of the previous proposition we have that $ \clz(L_\K(E))\cong \K$.

\end{corollary}
\begin{proof}
By Proposition~\ref{Paul} we have that $ \clz(L_\K(E))\cong C_v$, where $v$ can be choosen as the infinite emitter of Proposition~\ref{infinite1}. Since every element of $C_v$ is in the center of $v L_\K(E)v$ the results follows.
\end{proof}

So our task now is to consider row-finite prime Leavitt path algebras $A=L_\K(E)$ such that $E$ has a unique cycle with no exits $c$ and such that $E$ is not a comet. We give an example of this type of graph below.

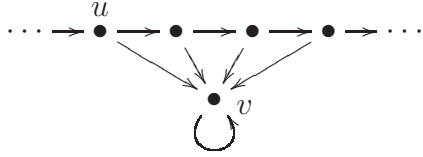
\[
\xygraph{
!{<0cm,0cm>;<1cm,0cm>:<0cm,1cm>::}
!{(0,.3)}*+{u}
!{(1.9,-1)}*+{v}
%!{(0.4,.6)}+{f_1}
!{(-1,0)}*+{\cdots}="a"
!{(0,0)}*+{\bullet}="b"
!{(1,0)}*+{\bullet}="c"
!{(2,0)}*+{\bullet}="d"
!{(3,0)}*+{\bullet}="g"
!{(4,0)}*+{\cdots}="e"
!{(1.5,-.9)}*+{\bullet}="f"
"a":"b" "b":"c" "c":"d"
"d":"g" "g":"e"
%%%%%%%%%%%%%%%
"b":"f" "c":"f" "d":"f" "g":"f"
%"f" :@(ul,ur) "f"
"f" :@(dl,dr) "f"
}\]

Before we characterize the centroid of $A$ we need a lemma.

%\begin{lemma}\label{capa}
%Suppose that $E$ is a row-finite graph, with a unique cycle  $c$ and assume that it has no exits. Let $H=\overline{c^0}$. Then for all $v, w \in H$ the cardinality of the set $\{\lambda \in \text{Path}(E): s(\lambda)=v, r(\lambda)=w, \text{ and $\lambda$ does not contain $c$} \}$ is finite.
%\end{lemma}
%\begin{proof}
%Recall that $\displaystyle H=\cup_{n\in \N}\Lambda^n(c^0)$. We prove the result using induction on $n$.

%For $n=0$ the result is clear. Assume it holds for $n$. Notice that if $u, v\in \Lambda^{n+1}(c^0)\setminus \Lambda^{n}(c^0)$ then there is no path between $u$ and $v$ and vice-versa:  if  there is a path between these two vertices, one of them  must belong to $\Lambda^n(c^0)$, a contradiction. 
%Now, if $u\in\Lambda^{n+1}(c^0)\setminus\Lambda^n(c^0)$ and $v\in\Lambda^n(c^0)$,
%for any $e\in s^{-1}(u)$ the number of paths from $r(e)$ to $v$ not containing $c$ is finite (by the induction hypothesis). This implies the desired result.
%\end{proof}

\begin{lemma}\label{todosLosCaminos}

 Let $E$ be an arbitrary graph, $H$ a hereditary set, and $v$ a vertex in the hereditary and saturated closure of $H$, but not in $H$. Then $$\Pi_{v,H}:=\{\lambda=f_1\cdots f_n\in\path(E)\vert s(\lambda)=v, r(\lambda)\in H, s(f_n)\notin H\}$$ is finite and hence we can write \begin{equation}\label{beepbeep} v=\sum_{\alpha\in\Pi_{v,H}}\alpha\alpha^\ast.
\end{equation}
\end{lemma}
\begin{proof}

%$\Pi_{v, H}$ as the set of paths starting at $v$ and finishing at a vertex in $H$, and defining $\Gamma:=\{\alpha\alpha^\ast \ \vert \alpha \in \Pi_{v, H}\}$, %
%We first observe that $\Pi_{v,H}$ is finite: 
We give a proof by induction: since $\overline{H}=\cup_{n\in\N }\Lambda^n(H)$ and $v\in\overline{H}\setminus H$, if $v\in\Lambda^1(H)$ then
$s^{-1}(v)=\{f_1,\ldots,f_n\}$ is finite because $v$ is a regular vertex.
Thus the number of paths from $v$ to $H$ in this case is $n$ and Equation~\ref{beepbeep} clearly holds.
Now assume that for $w\in\Lambda^n(H) \setminus H$ the set $\Pi_{w,H}$ is finite and $w=\sum_{\alpha\in\Pi_{w,H}}\alpha\alpha^\ast$.
Taking $v\in\Lambda^{n+1}(H)\setminus H$ we have that $r(s^{-1}(v))$ is
a finite subset of $\Lambda^n(H)$. Writing $r(s^{-1}(v))=\{u_1,\ldots,u_q,\ldots, u_k\}$ we may assume that $u_1,\ldots,u_q\in H$ while  $u_{q+1},\ldots,u_k\notin H$.
Then we can apply the induction hypothesis to each of $u_{q+1},\ldots, u_k$ and the final conclusion is that $\Pi_{v,H}$ is finite. 
Proceeding as in the proof of \cite[Lemma 1.2]{CMMSS} we conclude that
\begin{equation*} v=\sum_{\alpha\in\Pi_{v,H}}\alpha\alpha^\ast.
\end{equation*}
\end{proof}

To prepare for our next result, let 
$E$ be a row-finite graph that has a unique cycle $c$; assume that it has no exits and that $E$ is not a comet. For $v\notin H:=\overline{c^0}$ %(notice that such vertices exist since $E$ is not a comet) 
define $$\Gamma_1(v):=\{f\in E^1 \ \vert \ s(f)=v, r(f)\notin H\}, \text{ and }$$
$$\Gamma_2(v):=\{g\in E^1 \ \vert \ s(g)=v, r(g)\in H\}.$$
%Notice that $\Gamma_1$ and $\Gamma_2$ depend on the vertex $v$, but we do not include subscripts to keep notation lighter. 

\begin{lemma}\label{carita}
Let $A=L_\K(E)$ be a prime Leavitt path algebra such that $E$ has a unique cycle $c$; assume that it has no exits and that $E$ is not a comet. Let $\tau \in \clz(A)$. For $v\notin H:=\overline{c^0}$ write
$\tau(v)=kv+\sum_{f\in\Gamma_1(v)} f\xi_f f^*+\sum_{g\in\Gamma_2(v)}g\xi_g g^*$ as in Lemma~\ref{rap}. If $\xi_f=0$, for some $f\in \Gamma_1(v)$, then $\xi_g=0$ for every $g\in \Gamma_2(v)$.
\end{lemma}
\begin{proof}
Suppose that there exists $f_0\in \Gamma_1(v)$ such that $\xi_{f_0}=0$.

Given an arbitrary $h\in \Gamma_2(v)$, we know that $r(h)\in H$. We distinguish two possibilities:
\begin{enumerate}
    \item If $r(h)\in c^0$, since every vertex connects with $c^0$, there is a path $\l$ such that $s(\l)=r(f_0)$ and $r(\l)=r(h)$. Then defining $\nu:=f_0\l h^*$ and taking into account $\tau(v)\nu=\nu\tau(v)$ we get $\xi_h=0$.
    \item If $r(h)\notin c^0$, given any path $\mu\in\Pi_{r(h),c^0}$, since every vertex connects to $c$, there exists a path $\lambda$ such that $s(\lambda)= r(f_0)$,  and $r(\lambda)=r(\mu)$. Let $\nu:=f_0\lambda\mu^\ast h^\ast$. Then

$$k\nu = \tau(v)\nu=\nu\tau(v)= k\nu + f_0\lambda\mu^\ast\xi_hh^\ast.$$
This implies $\mu^\ast\xi_h=0$ and, consequently, $\mu\mu^\ast\xi_h=0$. Note that $\mu$ is an arbitrary path in $\Pi_{r(h),c^0}$. Since the graph is row finite, by Lemma \ref{todosLosCaminos}, we have that $r(h)=\sum_{\mu\in\Pi_{r(h),c^0}}\mu\mu^\ast$ and therefore
$$\xi_h=\sum_{\mu\in \Pi_{r(h),c^0}}\mu\mu^\ast\xi_h=0.$$
\end{enumerate}
\end{proof}

\begin{proposition}\label{kfield}
Let $A=L_\K(E)$ be a row-finite, prime Leavitt path algebra such that $E$ has a unique cycle $c$; assume that it has no exits and  that $E$ is not a comet. Then $\clz(A)\cong\K$.

\end{proposition}
\begin{proof}
Let $\tau \in \clz(A)$, take $v\notin H=\overline{c^0}$ (so $\Gamma_1(v)\neq \emptyset$) and write $\tau(v)=kv+\sum_{f\in\Gamma_1(v)} f\xi_f f^*+\sum_{g\in\Gamma_2(v)}g\xi_g g^*$ as in Lemma~\ref{rap}. We then have the following dichotomy: 
\begin{enumerate}
    \item  $\xi_f\ne 0$ for all $f\in\Gamma_1(v)$, or
    \item There exists $f\in\Gamma_1(v)$ such that $\xi_f=0$. 
\end{enumerate}

In the second case above, by the previous lemma, we have that $\xi_g=0$ for all $g\in\Gamma_2(v)$. Let $f_1\in \Gamma_1(v)$ with $\xi_{f_1}=0$. Then 
$\tau(r(f_1))=f_1^* \tau(v)f_1=k r(f_1)$, and hence applying Proposition~\ref{Paul}, with $u=r(f_1)$, we have $\clz(A)\cong\K$.

  Let us analyze the possibility (1). If this happens,
we have again a dichotomy:
\begin{enumerate}
    \item[(a)] For all $v'\in T(v)$ the element $\tau(v')$ is in case (1) above.
    \item[(b)] $\exists \ v'\in T(v)$ such that
    $\tau(v')$ satisfies possibility (2) above.
\end{enumerate}

If Condition (b) above is satisfied then we proceed as in case (2) and obtain that $\clz(A)\cong\K$. So we are left with the alternative where for all $v'\in T(v)$ the element $\tau(v')$ is in case (1). Let $v_1=r(e_1)$ where $e_1\in \Gamma_1(v)$. Since $v_1\notin H$, $\Gamma_1(v_1)\neq \emptyset$. Proceed inductively: Once $e_n$ and $v_n$ are defined, let $v_{n+1}=r(e_{n+1})$ where $e_{n+1}\in \Gamma_1(v_{n})$. Notice that each $v_n$ falls under alternative (1) above (since $v_n\in T(v)$ for all $n$). Furthermore the infinite path $e_1 e_2 \ldots$ do not connect to $H$. Now, notice that $\partial (\tau(v_i)) > \partial (\tau(v_{i+1})$ for all $i$ (notice the strict inequality). Indeed, if $$\tau(v_i) = kv_i+\sum_{f\in\Gamma_1(v_i)} f\xi_f f^*+\sum_{g\in\Gamma_2(v_i)}g\xi_g g^*$$ then $$\tau(v_{i+1})= \tau(r(e_{i+1}))= \tau (e_{i+1}^* v_i e_{i+1})= e_{i+1}^* \tau(v_i) e_{i+1} = k v_{i+1} + \xi_{e_{i+1}}.$$ But then $\partial (\tau(v))$ has to be infinity, a contradiction.
This concludes the proof.
\end{proof}

\section{Main Theorem}

We summarize the results of our paper below.

\begin{theorem} Let $E$ be a graph, $L_\K(E)$ the associated Leavitt path algebra, and denote by $ \clz(L_\K(E))$ the centroid of $L_\K(E)$. Then

\begin{enumerate}
    \item If $L_\K(E)$ is simple, then $ \clz(L_\K(E))\cong \K$ (see Proposition~\ref{pasteldenata}).
    \item If $L_\K(E)$ is row-finite, graded simple and  non-simple, then the graph $E$ is a comet and $ \clz(L_\K(E))\cong\K[x, x^{-1}]$ (see Proposition~\ref{PamDeQueijo}). %Furthermore, the extended centroid of $L_\K(E)$ is isomorphic to $\K(x)$ 
    \item If $L_\K(E)$ is prime and $E$ is acyclic, then $ \clz(L_\K(E))\cong \K$ (see Proposition~\ref{sol}).
    
    \item If $L_\K(E)$ is prime and there is a cycle with exits in $E$, then $ \clz(L_\K(E))\cong \K$ (see Corollary~\ref{Billy}). 
    
    \item If $L_\K(E)$ is prime and there is an infinite emitter in $E$ which is not the base of a cycle, then $ \clz(L_\K(E))\cong \K$ (see Corollary~\ref{infinite}).
    
    \item Suppose that $A=L_\K(E)$ is a row-finite prime Leavitt path algebra such that $E$ has a unique cycle (and this cycle has no exits). If $E$ is not a comet then $\clz(A)\cong\K$ (see Proposition~\ref{kfield} ), and if $E$ is a comet then $\clz(A)\cong \K[x, x^{-1}]$ (see Proposition~\ref{PamDeQueijo}). 
    
\end{enumerate}

\end{theorem}

A more compact form of the above statement is the following:

\begin{theorem}
Let $E$ be a graph, $L_\K(E)$ the associated Leavitt path algebra, and denote by $ \clz(L_\K(E))$ the centroid of $L_\K(E)$. If $L_\K(E)$ is prime, then $ \clz(L_\K(E))\cong \K$ except if $E$ is a row-finite comet, in which case $ \clz(L_\K(E))\cong \K[x,x^{-1}]$. %and the extended centroid of $L_\K(E)$ is $\K(x)$.
\end{theorem}
\begin{proof}
If $E$ is acyclic, then $ \clz(L_\K(E))\cong \K$ by Proposition~\ref{sol}. Otherwise $E$ has cycles. If there are  cycles with exits,  then $ \clz(L_\K(E))\cong \K$ applying Corollary~\ref{Billy}. So, we consider the case that $E$ has a unique cycle and it has no exit. If there is an infinite emitter then this is not the base of the cycle and, applying Corollary~\ref{infinite}, we have again that $ \clz(L_\K(E))\cong \K$. Therefore we assume in the sequel that $E$ is row finite. If $E$ is a comet, then 
$ \clz(L_\K(E))\cong \K[x,x^{-1}]$ by the proof of Proposition~\ref{PamDeQueijo} and, if $E$ is not a comet, then $ \clz(L_\K(E))\cong \K$ by Proposition~\ref{kfield}.

\end{proof}

\bibliographystyle{amsplain}

\end{document}